\documentclass[11pt]{article}        

\usepackage{etoolbox}

\usepackage{amsthm,amssymb}
\usepackage[shortlabels]{enumitem}
\usepackage{graphicx}

\usepackage{amscd,amssymb}
\usepackage{amsmath}
\usepackage{tikz}
\usepackage{tikz-3dplot}
\usepackage{tikz-cd}
\usepackage{faktor}

\newtheorem  {theorem}                  {Theorem}
\newtheorem* {theorem*}                   {Theorem}

\newtheorem {lemma}[theorem] {Lemma}
\newtheorem {prop}[theorem]      {Proposition}
\newtheorem* {prop*}     {Proposition}

\newtheorem {corollary}[theorem]      {Corollary}

\theoremstyle{definition}
\newtheorem {defi}[theorem] {Definition}
\newtheorem {Remark} [theorem]         {Remark}

\newtheorem* {Example*}    {Example}

\def\R{\mathbb{R}}

\newcommand{\norm}[1]{\|#1\|}

\newcommand{\T}{\mathbb{T}}

\newcommand{\BP}{\mathbb{P}}
\newcommand{\BR}{\mathbb{R}}

\title{KAM splittings and equidistributed periodic orbits for stable hypersurfaces}
\author{Robert Cardona, Oliver Edtmair and Rohil Prasad}
\date{}

\newcommand{\Addresses}{
{
  \bigskip

{\sc \noindent Robert Cardona}

\noindent Departament de Matem\`atiques i Inform\`atica, Universitat de Barcelona, Gran Via de Les Corts Catalanes 585, 08007 Barcelona, Spain; Centre de Recerca Matemàtica, Campus de Bellaterra, Edifici C, 08193, Barcelona, Spain.

{\noindent  \em robert.cardona@ub.edu\/}

\bigskip
{\sc \noindent Oliver Edtmair}

\noindent Institute for Advanced Study, 1 Einstein Drive, Princeton, NJ 08540, USA.

{\noindent \em oedtmair@ias.edu\/}

\bigskip
{\sc \noindent Rohil Prasad}

{\noindent \em r0hil.pras4d@gmail.com\/}
}

  \medskip

}

\begin{document}
\maketitle

\begin{abstract}

We show that any stable hypersurface of a symplectic $4$-manifold, on which the cohomology class of the symplectic form restricts to a multiple of a rational class, can be $C^\infty$-approximated by (possibly unstable) hypersurfaces whose closed characteristics equidistribute. The cohomological condition is necessary due to a famous example of Herman. The proof combines KAM theory with recent quantitative closing lemmas for Reeb flows and area-preserving maps. As a further application, we prove that every geodesible volume-preserving vector field on a closed three-manifold can be $C^\infty$-approximated by volume-preserving vector fields with equidistributed periodic orbits.

\end{abstract}

\section{Introduction}

The existence and distribution of periodic orbits are among the central questions in Hamiltonian dynamics. A particularly influential problem in this direction is the $C^\infty$-closing lemma, along the lines of the tenth problem of
Smale's list \cite{Smale}, which asks whether periodic trajectories can be created by arbitrarily small smooth perturbations of a given dynamical system. For Hamiltonian flows on hypersurfaces of symplectic manifolds, Fish and Hofer introduced in \cite{FH} a $C^\infty$-closing property asserting that periodic characteristics become dense after a (certain type of) generic perturbation of the hypersurface. While this property concerns the density of closed characteristics, recent progress on these questions has shown that one often has a quantitative ergodic refinement of density, namely, generic equidistribution of periodic orbits. 

For instance, in dimension three, generic density and equidistribution of periodic orbits were established for Reeb flows on contact manifolds by Irie \cite{Ir,Ir2}. Generic density of orbits was proven for Hamiltonian diffeomorphisms of closed surfaces in \cite{AI} and more generally for symplectomorphisms of closed surfaces in \cite{EH,CGPZ,CGPPZ}. This was followed by equidistribution results for closed surfaces \cite{P}, compact surfaces with boundary \cite{PP}, and punctured surfaces \cite{Z}. These results show that, after a generic perturbation, one can find collections of positively weighted periodic orbits that approximate the ambient measure (either the contact volume or the area form). The purpose of this paper is to investigate this phenomenon, from Fish--Hofer's perspective of hypersurfaces in four-dimensional symplectic manifolds, for the general class of \textit{stable hypersurfaces}, introduced by Hofer and Zehnder \cite{HZ} (see also Cieliebak--Volkov \cite{CV}). For other results on the Hamiltonian dynamics of stable hypersurfaces, see for example \cite{HT,CR}.

\paragraph{Equidistributed orbits near rational stable hypersurfaces.} Let $(W,\Omega)$ be a symplectic four-manifold, and let $\mathcal{HS}(W)$ denote the space of smooth, closed, cooriented embedded hypersurfaces in $W$, equipped with the $C^\infty$ topology. A hypersurface $M\in \mathcal{HS}(W)$ has an induced Hamiltonian structure $\omega$ (a closed two-form of maximal rank), obtained by the pullback of $\Omega$ by the inclusion map of $M$ in $W$. The line field $\ker \omega$ induces a one-dimensional foliation on $M$ called the \emph{characteristic foliation}.

We say that $M$ has \emph{dense closed characteristics} if the union of the closed leaves of its characteristic foliation is dense in $M$. A quantitative refinement is to ask for \emph{equidistributed closed characteristics}: roughly speaking, finite collections of positively weighted closed characteristics that converge, in the sense of currents, to the Hamiltonian structure $\omega$. We refer to Section~\ref{ss:currents} for a precise definition.

\begin{defi}\label{def:localclosing}
A hypersurface $M$ satisfies the \emph{neighboring $C^\infty$-closing property} if there is a neighborhood $\mathcal{V}\subset \mathcal{HS}(W)$ of $M$ and a residual subset $\mathcal{A}\subset \mathcal{V}$ such that every $\tilde M\in\mathcal{A}$ has dense closed characteristics. It satisfies the \emph{neighboring equidistribution property} if there are such $\mathcal{V}$ and $\mathcal{A}$ for which every $\tilde M\in\mathcal{A}$ has equidistributed closed characteristics.
\end{defi}

These properties are intrinsic to the Hamiltonian structure induced on $M$. Indeed, Gotay's coisotropic neighborhood theorem \cite[Local Uniqueness Theorem]{Gotay} implies that the symplectic neighborhood of $M$ is determined, up to symplectomorphism, by $\omega$. Thus the properties do not depend on the ambient symplectic manifold, and we may also speak of a Hamiltonian structure satisfying them.

\begin{Remark}
Fish and Hofer's $C^\infty$-closing property \cite[Definition 1.10]{FH} requires that, for some symplectic form on $M\times\mathbb{R}$ restricting to $\omega$ along the zero section, generic graphs of functions $f\in C^\infty(M)$ have dense closed characteristics. The neighboring $C^\infty$-closing property is the corresponding local formulation, requiring this only for generic hypersurfaces in a neighborhood of $M$. We use the term ``neighboring'' because Fish and Hofer use ``local $C^\infty$-closing property'' for the closing property of the standard symplectic space.
\end{Remark}

A Hamiltonian structure $\omega$ on $M$ is called \emph{rational} if $[\omega]\in H^2(M;\mathbb{R})$ is a real multiple of a rational cohomology class. In this language, the neighboring equidistribution property is satisfied by Hamiltonian structures of contact type \cite{Ir2} and by rational Hamiltonian structures arising from symplectic mapping tori \cite{P}. Both are examples of stable hypersurfaces. A hypersurface is \emph{stable} if its induced Hamiltonian structure admits a stabilizing one-form; see Section~\ref{ss:HamSt}. The main result of this work is the following simultaneous extension of these two results.
\begin{theorem}\label{thm:main}
Let $M$ be a stable hypersurface in $(W,\Omega)$ such that the induced Hamiltonian structure is rational. Then, arbitrarily $C^\infty$-close to $M$, there exists a stable hypersurface that satisfies the neighboring equidistribution property.
\end{theorem}

A particular consequence of this is that any rational stable hypersurface can be $C^\infty$-approximated by hypersurfaces with equidistributed closed characteristics. The rationality condition cannot be removed, as shown by Herman's counterexample to the $C^\infty$-closing lemma \cite{Herm1,Herm2}, which is a symplectic mapping torus (and thus a stable hypersurface). The notion of rationality has played a key role in several recent results on area-preserving diffeomorphisms \cite{CGPZ, EH, Pr, GLP}.

\paragraph{KAM splittings of Hamiltonian structures.} The proof of Theorem~\ref{thm:main} relies on the notion of a \textit{KAM splitting} of a Hamiltonian structure. Given a Hamiltonian structure $\omega$ on $M$, a KAM splitting is, roughly speaking, a finite collection of invariant KAM tori $T_1,\ldots,T_r$ that divide $M$ into domains that are either of contact type or symplectic mapping tori. We refer to Definition~\ref{def:KAMsplit} for details. We call a KAM splitting \textit{twist} if the Birkhoff normal form of the Hamiltonian structure along each splitting torus is non-degenerate in the KAM sense; see Section~\ref{ss:Birkhoff}. An essential step in proving Theorem~\ref{thm:main} is the following result.

\begin{theorem}\label{thm:KAMequi}
Let $M$ be an embedded hypersurface in a symplectic manifold $(W,\Omega)$ of dimension four, and let $\omega$ denote the Hamiltonian structure induced on $M$. If $(M,\omega)$ admits a twist KAM splitting and $\omega$ is rational, then it satisfies the neighboring equidistribution property.
\end{theorem}

The splitting allows us to apply the equidistribution results for Reeb flows \cite{Ir2} and symplectic mapping tori of surfaces with boundary \cite{PP} separately on each domain. The difficulty is to combine the resulting perturbations into a smooth perturbation on all of $M$: they need not vanish near the splitting tori, and cutting them off can destroy the closed characteristics used to approximate the Hamiltonian structure.

Two properties of twist KAM tori resolve this difficulty. Stability (Theorem~\ref{thm:KAMstable}) provides additional invariant tori arbitrarily close to the splitting tori, and robustness (Theorem~\ref{thm:KAMrobust}) ensures that these persist under sufficiently small cohomologous perturbations. These tori separate the closed characteristics needed for approximation away from the splitting tori from the regions where we cut off the perturbations. We can therefore combine the perturbations while retaining those characteristics, and control the remaining approximation error near the splitting tori.
\medskip

To deduce Theorem \ref{thm:main} from Theorem~\ref{thm:KAMequi}, we then give a sufficient condition for a stable hypersurface to admit a KAM splitting (Theorem \ref{thm:StRobust}). This step requires the use of tools in stable Hamiltonian topology \cite{CV}. Finally, one can show that this condition can be achieved by a $C^\infty$-perturbation starting from an arbitrary stable hypersurface. 

\begin{Remark}
From the main results in \cite{CR}, one can deduce that any stable hypersurface can be $C^\infty$-perturbed to another stable hypersurface whose characteristic foliation admits a Birkhoff section. More precisely, it follows from \cite[Theorem B and Remark 6.5]{CR} that any stabilizable Hamiltonian structure can be $C^\infty$-perturbed to admit a Birkhoff section. By \cite[Remark 6.3]{CR} this perturbation can be taken exact, which further implies that it can be realized by a corresponding hypersurface perturbation \cite[Lemma 49]{C}.

One might therefore ask whether Theorem \ref{thm:main} can be deduced by applying the generic equidistribution result for area-preserving diffeomorphisms of compact surfaces with boundary \cite{PP} to the first-return map of the Birkhoff section. However, there are a few important difficulties with this approach. For instance, the first-return map of a Birkhoff section is in general only an area-preserving diffeomorphism of the interior of the Birkhoff section. While the first-return map does extend to a diffeomorphism of the surface with boundary if the Birkhoff section is $\partial$-strong (see e.g. \cite{FlHr25}), the area form vanishes on the boundary of the surface, so the results of \cite{PP} are not directly applicable. Another difficulty is the slightly subtle question of which perturbations of the first-return map (not supported in the interior of the Birkhoff section) can be realized by a perturbation of the Hamiltonian structure (see e.g. \cite{AGZ22}).
\end{Remark}

\paragraph{Conservative 3D flows.} Let $M$ be a closed three-manifold equipped with a volume form $\mu$, and let $\mathfrak{X}_\mu(M)$ denote the space of smooth nowhere-vanishing vector fields preserving $\mu$, equipped with the $C^\infty$ topology. A vector field $X\in\mathfrak{X}_\mu(M)$ has \emph{equidistributed periodic orbits} if the Hamiltonian structure $\omega=\iota_X\mu$ has equidistributed closed characteristics in the sense of Definition~\ref{def:equiHam}.

Let $\mathcal{SR}_\mu(M)\subset\mathfrak{X}_\mu(M)$ denote the subset of vector fields parallel to the Reeb field of a stable Hamiltonian structure. Equivalently, $\mathcal{SR}_\mu(M)$ consists of the vector fields $X\in\mathfrak{X}_\mu(M)$ for which $\iota_X\mu$ is stabilizable. These are precisely the $\mu$-preserving geodesible vector fields, where geodesibility means that the orbits can be made geodesics of a Riemannian metric after reparameterization \cite{Re}.

\begin{theorem}\label{thm:main2}
    Let $M$ be a closed three-manifold and $\mu$ a volume form. For any vector field $X\in \mathcal{SR}_\mu(M)$, there is an arbitrarily $C^\infty$-close vector field $Y\in \mathfrak{X}_\mu(M)$ with equidistributed periodic orbits.
\end{theorem}

Theorem~\ref{thm:main2} does not follow directly from Theorem~\ref{thm:main} because $\iota_X\mu$ need not be rational. We adapt its proof by allowing perturbations that change the cohomology class in the regions where rationality is needed; see Section~\ref{ss:app3D}. The approximating vector field $Y$ need not belong to $\mathcal{SR}_\mu(M)$.

The conclusion was previously known for vector fields of contact type, meaning that $\iota_X\mu$ admits a contact primitive \cite{Ir2}, and for vector fields admitting a global cross section \cite{AI, CGPZ, EH, CGPPZ}; see Section~\ref{ss:app3D}. However, the class $\mathcal{SR}_\mu(M)$ is notably larger: for instance, there exist vector fields in $\mathcal{SR}_\mu(M)$ with a $C^1$-neighborhood (and even $C^0$-neighborhood) in $\mathfrak{X}_\mu(M)$ containing neither vector fields of contact type nor vector fields admitting a global cross section. For the $C^1$-topology, this follows from \cite[Propositions 22 and 27 and Corollary 32]{C}. For the $C^0$-topology, it is not difficult to construct a vector field $X$ in $\mathcal{SR}_\mu(M)$ such that $\iota_X\mu$ is exact, say with primitive $\beta$, and such that there is a pair of null-homologous periodic orbits with positive and negative $\beta$-integrals. This property obstructs $X$ from being contact type and persists under $C^0$-perturbations \cite[Theorem 15 and its proof]{CTdL}, and also obstructs the existence of a global cross-section by Schwartzman's characterization \cite{Sch}. 

\paragraph{Organization of the paper.}
Section~\ref{sec:prelimi} collects preliminary definitions and results. In Section~\ref{sec:equi_KAM}, we introduce KAM splittings and prove Theorem~\ref{thm:KAMequi} using the KAM results established in Section~\ref{sec:KAM_proofs}. Section~\ref{sec:equi_stable} contains the proofs of Theorems~\ref{thm:main} and~\ref{thm:main2}. Finally, Section~\ref{sec:KAM_proofs} develops the KAM theory for Hamiltonian structures.

\paragraph{Acknowledgments.} RC acknowledges partial support from the AEI grant PID2023-147585NA-I00, the Departament de Recerca i Universitats de la Generalitat de Catalunya (2021 SGR 00697), and the Spanish State Research Agency, through the Severo Ochoa and María de Maeztu Program for Centers and Units of Excellence in R\&D (CEX2020-001084-M). OE is supported by Dr. Max R\"ossler, the Walter Haefner
Foundation, and the ETH Z\"urich Foundation. This research was partially conducted during the period OE served as a Clay Research Fellow. RP acknowledges support from the Miller Institute for Basic Research at UC Berkeley, as well as the warm hospitality of the University of Barcelona, where part of this research was conducted. 

\paragraph{AI disclosure.} AI models were only used for proofreading assistance during the final stages of manuscript preparation.

\section{Preliminaries}\label{sec:prelimi}

In this section, we recall some preliminary notions needed for the proofs of our main results.

\subsection{Area-preserving surface maps} Let $(\Sigma, \tau)$ denote a compact symplectic surface, possibly with non-empty boundary. A diffeomorphism $\phi: \Sigma \to \Sigma$ is \emph{area-preserving} if $\phi^*\tau = \tau$. 

\subsubsection{Mapping tori}

Let $\phi$ be an area-preserving diffeomorphism. The \emph{mapping torus}
$$Y_\phi := [0,1] \times \Sigma/\{(1, x) \sim (0, \phi(x))\}$$
is obtained by gluing the right boundary component $\{1\}\times \Sigma$ of $[0,1] \times \Sigma$ onto the left boundary component $\{0\}\times \Sigma$ via $\phi$. This is a smooth, compact $3$-manifold fibering over the circle with fiber $\Sigma$.

Letting $t$ denote the coordinate on $[0,1]$, the $1$-form $dt$ and vector field $\partial_t$ descend to the manifold $Y_\phi$. The flow of $\partial_t$ on $Y_\phi$ can be thought of as a continuous-time version of the map $\phi$. The $2$-form $\tau$ descends to a closed $2$-form $\tau_\phi$ on $Y_\phi$. The line field $\operatorname{ker}\tau_\phi$ is spanned by $\partial_t$. In the terminology of Section~\ref{ss:HamSt} below, the $2$-form $\tau_\phi$ is a Hamiltonian structure on $Y_\phi$ and $dt$ is a stabilizing $1$-form.

\subsubsection{Hamiltonian isotopy} 

Let $\phi$ be an area-preserving diffeomorphism and let $Y_\phi$ be its mapping torus. Let $H: Y_\phi \to \mathbb{R}$ be a smooth function such that, for any fiber $F \subset Y_\phi$, the restriction $H|_F$ is locally constant on the boundary of $F$. For such $H$, a new area-preserving diffeomorphism $\phi^H$ can be constructed as follows. Pull back $H$ to view it as a time-dependent Hamiltonian $H: [0,1] \times \Sigma \to \mathbb{R}$. The Hamiltonian vector field $X_H$ is defined by the equation $\tau(X_{H_t}, -) = dH_t$. Let $\psi^t_H$ denote the time-$t$ flow of $X_H$. By the boundary condition for $H$, the vector field $X_H$ is tangent to $\partial\Sigma$ at all times, so $\psi^t_H$ is well-defined for all times. Define $\phi^H := \phi \circ \psi^1_H$.

Two area-preserving diffeomorphisms $\phi$ and $\phi'$ are \emph{Hamiltonian isotopic} if there exists $H: Y_\phi \to \mathbb{R}$ such that $\phi' = \phi^H$. It is not difficult to see that Hamiltonian isotopy is an equivalence relation. A single area-preserving diffeomorphism $\phi$ is called \emph{Hamiltonian} if it is Hamiltonian isotopic to the identity map. 

Suppose that $\phi$ and $\phi'$ are Hamiltonian isotopic. Then, any choice of $H$ such that $\phi' = \phi^H$ induces a diffeomorphism
\begin{equation}\label{eq:mt_bijection}
\begin{split}
f_H: Y_\phi \to Y_{\phi'}, \\
[(t,x)] \mapsto [(t, (\psi^t_H)^{-1}(x))].
\end{split}
\end{equation}
The following identity will be useful shortly:
\begin{equation}\label{eq:mt_pullback}f_H^*\tau_{\phi'} = \tau_\phi + dH \wedge dt.\end{equation}

\subsubsection{Rational area-preserving diffeomorphisms}
\label{subsubsec:rational_area_pres_diffeo}

An area-preserving diffeomorphism $\phi$ is called \emph{rational} if the cohomology class 
\begin{equation*}
[\tau_\phi] \in H^2(Y_\phi; \mathbb{R})
\end{equation*}
is a real multiple of a class in $H^2(Y_\phi; \mathbb{Q})$. Rationality is invariant under Hamiltonian isotopy. This claim follows from equation~\eqref{eq:mt_pullback}.

\begin{Remark}
\label{rem:different_notions_of_rationality}
    The notion of rationality we adopt in our paper is slightly weaker than the notion of rationality used in~\cite{PP}, whose results we are going to use. In that paper, it is required that $A^{-1}[\tau_\phi]$ is a rational cohomology class, where $A$ is the total area of $\Sigma$. For closed surfaces, these two notions of rationality agree, but for surfaces with non-empty boundary our notion is more general.
\end{Remark}

\begin{lemma}
\label{lem:rational_extension_to_closed_surface}
    Suppose that $\Sigma$ is a compact surface with non-empty boundary and that $\phi$ is a rational area-preserving diffeomorphism of $\Sigma$. Then there exist an area-preserving embedding of $\Sigma$ into a closed surface $\hat\Sigma$ and an extension $\hat\phi$ of $\phi$ to a rational area-preserving diffeomorphism of $\hat{\Sigma}$.
\end{lemma}

\begin{proof}
    We first build a surface $\Sigma'$ by attaching annuli to the boundary components of $\Sigma$. We choose the areas of these annuli such that $(A')^{-1}[\tau_\phi]$ is a rational cohomology class, where $A'$ is the total area of $\Sigma'$. It follows from~\cite[Prop. 4.1]{PP} that $\phi$ extends to an area-preserving diffeomorphism $\phi'$ of $\Sigma'$. Strictly speaking, the cited proposition attaches disks instead of annuli, but the arguments go through with only minor notational modifications. We observe that $(A')^{-1}[\tau_{\phi'}]$ is a rational cohomology class. In other words, the diffeomorphism $\phi'$ is rational in the stronger sense of~\cite{PP}, see Remark~\ref{rem:different_notions_of_rationality}.

    The closed surface $\hat\Sigma$ is obtained from $\Sigma'$ by capping off all boundary circles by disks. Again using~\cite[Prop. 4.1]{PP}, we can extend $\phi'$ to an area-preserving diffeomorphism $\hat\phi$ of $\hat\Sigma$. By~\cite[Lem. 4.5]{PP}, we can adjust the areas of the capping disks such that $\hat\phi$ is rational.
\end{proof}

\subsubsection{Periodic orbits and closed characteristics}

Let $\phi$ be an area-preserving diffeomorphism. A \emph{simple periodic orbit} is a finite ordered subset $\mathbf{x} = \{x_1, \ldots, x_k\} \subset \Sigma$ that is cyclically permuted by $\phi$. Two simple periodic orbits are considered equivalent if the underlying point sets agree. Let $\mathcal{P}(\phi)$ denote the set of equivalence classes of simple periodic orbits. 

A \emph{closed characteristic} is an oriented embedded loop $c$ in the mapping torus $Y_\phi$ that is everywhere positively tangent to the vector field $\partial_t$. Let $\mathcal{C}(\phi)$ denote the set of closed characteristics. 

Any $\mathbf{x} \in \mathcal{P}(\phi)$ induces a closed characteristic $c \in \mathcal{C}(\phi)$. Write $\mathbf{x}$ as a set $\{x_1, \ldots, x_k\} \subset \Sigma$. Then, the corresponding closed characteristic $c$ is given by the image in $Y_\phi$ of the collection of segments 
$$\bigcup_{i=1}^k [0,1] \times x_i \subset [0,1] \times \Sigma.$$

One can check that the map $\mathbf{x} \mapsto c$ is a bijection between $\mathcal{P}(\phi)$ and $\mathcal{C}(\phi)$. An inverse map $\mathcal{C}(\phi) \to \mathcal{P}(\phi)$ is constructed as follows. For each $c \in \mathcal{C}(\phi)$, the intersection of $c$ with $\{0\} \times \Sigma \subset Y_\phi$ is a simple periodic orbit $\mathbf{x}$. The cyclical ordering on $\mathbf{x}$ is determined by the orientation of $c$. 

\subsubsection{Currents}

Let $\mathcal{P}_{\mathbb{R}}(\phi)$ denote the set of finite formal sums of simple periodic orbits, with positive real weights. Elements of $\mathcal{P}_{\mathbb{R}}(\phi)$ are naturally \emph{$0$-dimensional currents}, i.e. continuous linear functionals on $C^\infty(\Sigma)$. For a single periodic orbit $\mathbf{x}$ and a smooth function $f$, define $\mathbf{x}(f)$ to be the sum of the values of $f$ over the points in $\mathbf{x}$. For a general $\mathbf{x} \in \mathcal{P}_{\mathbb{R}}(\phi)$, expand $\mathbf{x} = \sum a_i \mathbf{x}_i$ and define $\mathbf{x}(f) := \sum a_i \mathbf{x}_i(f)$. 

Let $\mathcal{C}_{\mathbb{R}}(\phi)$ denote the set of finite formal sums of closed characteristics, with positive real weights. Elements of $\mathcal{C}_{\mathbb{R}}(\phi)$ are called \emph{characteristic currents}. Indeed, they are naturally \emph{$1$-dimensional currents}, i.e. continuous linear functionals on the space of smooth $1$-forms on $Y_\phi$. For a single closed characteristic $c$ and a $1$-form $\alpha$, define $c(\alpha) := \int_c \alpha$. For a general $c \in \mathcal{C}_\R(\phi)$, expand $c = \sum a_i c_i$ and define $c(\alpha) := \sum a_i c_i(\alpha)$. 

It is clear that the bijection $\mathcal{P}(\phi) \to \mathcal{C}(\phi)$ described above extends to a bijection $\mathcal{P}_{\mathbb{R}}(\phi) \to \mathcal{C}_{\mathbb{R}}(\phi)$. Let us also record how the bijection acts on the level of currents. Consider $\mathbf{x} \in \mathcal{P}_{\mathbb{R}}(\phi)$ and $c \in \mathcal{C}_{\mathbb{R}}(\phi)$ corresponding to each other under the bijection. Given a smooth $1$-form $\alpha$ on $Y_\phi$, let $f_\alpha: \Sigma \to \mathbb{R}$ be the smooth function defined by lifting $\alpha$ to $[0,1] \times \Sigma$ and then integrating along the segments $[0,1] \times \{x\}$. Then, observe that
\begin{equation*} c(\alpha) = \mathbf{x}(f_\alpha).\end{equation*}
Note that every smooth function $f:\Sigma \rightarrow \mathbb{R}$ arises as $f_\alpha$ for some $1$-form $\alpha$ on $Y_\phi$.

\subsection{Reeb flows on $3$-manifolds} Let $M$ be a compact three-manifold, possibly with non-empty boundary. A \emph{contact form} on $M$ is a $1$-form $\alpha$ such that $\alpha \wedge d\alpha \neq 0$. If $M$ has non-empty boundary, we require that $\ker (d\alpha)$ is tangent to $\partial M$. The \emph{Reeb vector field} $R$ is implicitly defined by the equations
$$\alpha(R) \equiv 1,\qquad d\alpha(R, -) \equiv 0.$$
Observe that $R$ is a nowhere-vanishing section of $\ker d\alpha$.

\subsection{Hamiltonian structures on $3$-manifolds}\label{ss:HamSt}

Let $M$ be a compact oriented three-manifold, possibly with non-empty boundary.

\subsubsection{Basics} A \emph{Hamiltonian structure} on $M$ is a closed $2$-form $\omega$ such that $\ker\omega$ is $1$-dimensional at each point. If $M$ has non-empty boundary, we require that $\ker\omega$ is tangent to $\partial M$. The kernel of $\omega$ is a $1$-dimensional subbundle of $TM$ that integrates to a $1$-dimensional foliation on $M$, called the \emph{characteristic foliation}. The characteristic foliation carries a natural orientation induced by the ambient orientation of $M$ and the transverse orientation determined by $\omega$. The line field $\ker\omega$ is therefore always trivial as a line bundle, so it admits nowhere-vanishing sections. Any such section is a nowhere-vanishing vector field, and this vector field preserves some smooth volume form on $M$. There is a canonical choice of such a section after we fix a \emph{framing}. This is a $1$-form $\lambda$ such that $\lambda \wedge \omega > 0$. The pair $(\lambda, \omega)$ is called a \emph{framed Hamiltonian structure}, and its \emph{Reeb vector field} $R$ is implicitly defined by the equations
$$\lambda(R) \equiv 1,\qquad \omega(R, -) \equiv 0.$$
Its flow preserves the volume form $\lambda \wedge \omega$.  

Hamiltonian structures are naturally induced on hypersurfaces in symplectic manifolds. An embedded hypersurface $M$ in a symplectic manifold $(W,\Omega)$ inherits a Hamiltonian structure $\omega = \Omega|_M$. The characteristic foliation of $\omega$ encodes the dynamics of Hamiltonian flows on $M$. To elaborate, choose a smooth function $H: W \to \mathbb{R}$ such that $M$ is a regular component of the level set $H^{-1}(0)$, i.e. $H$ has no critical points on $M \subset H^{-1}(0)$. Then, the Hamiltonian vector field $X_H$ is tangent to $M$, and its restriction to $M$ is tangent to $\ker\omega$. Therefore, the integral curves of $X_H$ on $M$ coincide with the leaves of the characteristic foliation on $M$. \medskip

A Hamiltonian structure is \emph{stabilizable} if there exists a $1$-form $\lambda$ such that $\lambda \wedge \omega > 0$ and $\ker\omega \subseteq \ker d\lambda$. The $1$-form is called a \emph{stabilizing form} and the pair $(\lambda, \omega)$ is called a \emph{stable Hamiltonian structure}. Note that $\lambda$ is a particular kind of framing, and that a stable Hamiltonian structure is a particular kind of framed Hamiltonian structure. Also, observe that, in three dimensions, there exists a unique smooth function $f$ such that $d\lambda = f\omega$. The Reeb vector field $R$ of a stable Hamiltonian structure is defined as above.

Both three-dimensional Reeb flows and area-preserving diffeomorphisms can be viewed as Reeb flows of a stable Hamiltonian structure. Given a contact form $\alpha$ on a compact three-manifold $M$, oriented by $\alpha\wedge d\alpha$, the pair $(\alpha, d\alpha)$ is a stable Hamiltonian structure. The Reeb vector field of $\alpha$ coincides with the Reeb vector field of $(\alpha, d\alpha)$. Given an area-preserving diffeomorphism $\phi$ of a compact surface $\Sigma$, the pair $(dt, \tau_\phi)$ is a stable Hamiltonian structure on $Y_\phi$. The Reeb vector field is equal to $\partial_t$. 

From the viewpoint of Hamiltonian structures, these situations can be described by contact-type Hamiltonian structures and Hamiltonian structures with global cross section, respectively. Let $\omega$ be a Hamiltonian structure on $M$. First, $\omega$ is \emph{contact-type} if $\omega = d\alpha$ for some contact form $\alpha$. Next, a \emph{global cross section} of $\omega$ is a neatly embedded compact surface $\Sigma \subset M$ that intersects transversely every leaf of the characteristic foliation of $\omega$. In this case, there exists a symplectic form $\tau$ on $\Sigma$, an area-preserving diffeomorphism $\phi$, and a diffeomorphism $Y_\phi \to M$ that pulls back $\omega$ to $\tau_\phi$. The flow on $M$ generated by any nowhere-vanishing section of $\ker\omega$ is orbit equivalent to the suspension flow of $\phi$.

\medskip

A Hamiltonian structure $\omega$ on a $3$-manifold $M$ is called \emph{rational} if the cohomology class $[\omega] \in H^2(M;\R)$ is a real multiple of a rational cohomology class. This is compatible with the notion of a rational area-preserving surface diffeomorphism introduced in Section~\ref{subsubsec:rational_area_pres_diffeo}. Indeed, an area-preserving diffeomorphism $\phi$ is rational if and only if the Hamiltonian structure $\tau_\phi$ on the mapping torus $Y_\phi$ is rational. Note that a Hamiltonian structure of contact type is automatically rational.

\subsubsection{Closed characteristics, characteristic currents and equidistribution}\label{ss:currents}

Let $\omega$ be a Hamiltonian structure on $M$. A \emph{closed characteristic} is an oriented embedded loop $c \subset M$ that is positively tangent to $\ker \omega$. Let $\mathcal{C}(\omega)$ denote the set of closed characteristics. Let $\mathcal{C}_{\mathbb{R}}(\omega)$ denote the set of finite formal sums of closed characteristics, with positive real weights. Elements of $\mathcal{C}_{\mathbb{R}}(\omega)$ are called \emph{characteristic currents}. They are naturally \emph{$1$-dimensional currents}, i.e. continuous linear functionals on the space of smooth $1$-forms on $M$. For a single closed characteristic $c$ and a $1$-form $\alpha$, define $c(\alpha) := \int_c \alpha$. For a general $c \in \mathcal{C}_\R(\omega)$, expand $c = \sum a_i c_i$ and define $c(\alpha) := \sum a_i c_i(\alpha)$. 

\begin{defi}\label{def:equiHam}
    A Hamiltonian structure $\omega$ on $M$ has equidistributed closed characteristics if there is a sequence $c_k \in \mathcal{C}_{\mathbb{R}}(\omega)$ such that $c_k \longrightarrow \omega$ as one-dimensional currents in the weak$^*$ topology. 
\end{defi}

Convergence in the weak$^*$ topology means that for any one-form $\beta\in \Omega^1(M)$, one has $c_k(\beta) \longrightarrow \int_M \beta \wedge \omega$.
Similarly, a hypersurface $M\subset (W,\Omega)$ has equidistributed closed characteristics if the induced Hamiltonian structure does.

\subsubsection{Structure of stable Hamiltonian $3$-manifolds}

Let $(\lambda,\omega)$ be a stable Hamiltonian structure on a closed $3$-manifold $M$. Let $f \in C^\infty(M)$ denote the unique smooth function satisfying $d\lambda = f\omega$. Let $R$ denote the Reeb vector field. We recall here the notion of a \emph{structural decomposition}, following \cite{CR}, and introduce the notion of an \emph{improved structural decomposition}. To prepare, we define the different pieces of the decomposition, starting with contact and fibered regions. 

\begin{defi}\label{def:contact_and_fibered}
    Let $N \subset M$ be a compact embedded $3$-dimensional submanifold, possibly disconnected and possibly with non-empty boundary. Assume further that if $\partial N$ is non-empty, then $\ker\omega$ is tangent to $\partial N$. Call $N$ a \emph{contact region} if $f$ is nowhere-vanishing on $N$. Call $N$ a \emph{fibered region} if there exists a smooth, closed $1$-form $\beta$ on $N$ such that $\beta \wedge \omega$ is nowhere vanishing.  
\end{defi}

The following pair of lemmas is immediate from Definition~\ref{def:contact_and_fibered}.

\begin{lemma}\label{lem:lambda_contact_region}
    Let $N \subset M$ be a contact region. Then $\lambda$ restricts to a contact form on $N$, and the Reeb vector field of $\lambda$ is equal to $R$. 
\end{lemma}

\begin{lemma} 
Let $N \subset M$ be a fibered region. Then there exists a smooth fibration of $N$ over the circle such that each fiber is compact, transverse to $\partial N$, and is a global cross section of $N$. 
\end{lemma}

Next, we define integrable regions. For the remainder of the paper, let $T^2 := \mathbb{C}/\left(\mathbb{Z}+\sqrt{-1}\mathbb{Z}\right)$ denote the standard $2$-torus, and let $z = x + y\sqrt{-1}$ denote the standard coordinate. 

\begin{defi}
    An \emph{integrable region} of $(\lambda,\omega)$ is the datum of an embedded compact and connected $3$-manifold $U \subset M$ with non-empty boundary, and an orientation-preserving diffeomorphism $\chi: U \to T^2 \times I$, where $I=[0,1]$, such that if $(x, y, t)$ denote the coordinates on $U$ induced by $\chi$, we have 
        $$\omega = h_1(t)dt \wedge dx + h_2(t)dt \wedge dy \quad \text{and} \quad \lambda = g_1(t)dx + g_2(t)dy + g_3(t)dt.$$
\end{defi}

For notational convenience, we often omit the coordinate map $\chi$ when discussing an integrable region. It follows from the definition that each integrable region smoothly fibers into Reeb-invariant tori. The dynamics on each torus may be described as follows. Let $S^1 \subset \mathbb{C}$ denote the unit circle. Given an integrable region $U$, let its \emph{slope} be the circle-valued function 
$$k_U := \frac{-h_2 +  h_1\sqrt{-1}}{|-h_2 + h_1\sqrt{-1}|}.$$ 
Note that, since $\omega$ is a Hamiltonian structure, at least one of $h_1$ and $h_2$ is nonzero at each point, so the function $k_U$ is well-defined. Now, we state a precise lemma describing the dynamics on $U$. 

\begin{lemma}\label{lem:integrable_regions}
    Let $U \subset M$ be an integrable region. For each $t \in I$, the torus $T^2 \times \{t\}$ is invariant under the Reeb flow. Moreover, the slope function $k_U$ is constant on $T^2 \times \{t\}$, and the Reeb vector field on $T^2 \times \{t\}$ is orbit-equivalent to the constant vector field $k_U|_{T^2\times \{t\}}$. 
\end{lemma}

\begin{defi}\label{def:struc}
    A \emph{structural decomposition} of a stable Hamiltonian $3$-manifold $(M,\lambda,\omega)$ is a triple $(N_0, N_c, N_{int})$ consisting of a fibered region $N_0$, a contact region $N_c$, and a finite disjoint union $N_{int}$ of integrable regions, each possibly empty, such that the following conditions hold: 
    \begin{enumerate}
        \item $N_0$ and $N_c$ are disjoint;
        \item $M = \operatorname{Int}(N_0) \cup \operatorname{Int}(N_c) \cup \operatorname{Int}(N_{int})$;
        \item for each connected component $U \simeq T^2 \times I$ of $N_{int}$, the complement $U\,\setminus\,(N_0 \sqcup N_c)$ is given by $T^2 \times (a, b)$ for some $0< a <b <1$ depending on $U$;
        \item for each connected component $U$ of $N_{int}$, the restriction of $f$ to $U\setminus(N_0\sqcup N_c)=T^2\times(a,b)$ has no critical points and is constant on each torus $T^2\times\{t\}$ with $t\in(a,b)$.
    \end{enumerate}

    The structural decomposition is an \emph{improved structural decomposition} if moreover 
    \begin{enumerate}
        \setcounter{enumi}{4}
        \item $f$ is locally constant on $N_c$,
        \item and $f$ is equal to $0$ on $N_0$. 
    \end{enumerate}
\end{defi}

We refer to \cite{CR} for a discussion of several aspects of this definition, and only mention those relevant to us. First, the definition is motivated by the theory developed in the seminal work of Cieliebak and Volkov \cite{CV}, which implies that any stable Hamiltonian structure admits a structural decomposition (concretely, this follows from \cite[Proposition 3.23 and Theorem 3.3]{CV}). Such a decomposition is, in general, non-unique. In addition, for a given stabilizable $\omega$, different choices of stabilizing one-forms give stable Hamiltonian structures with a priori different possible structural decompositions. 

The structure theorem \cite[Theorem 4.1]{CV}, in the formulation of \cite[Theorem 2.4]{CR}, asserts that after making a $C^1$-small perturbation to $\lambda$, there exists an improved structural decomposition. 

\begin{theorem}[Cieliebak-Volkov {\cite{CV}}]
\label{thm:struc}
    Let $(\lambda,\omega)$ be a stable Hamiltonian structure. In any $C^1$-neighborhood of $\lambda$, there exists a stabilizing $1$-form $\tilde\lambda$ such that $(\tilde\lambda, \omega)$ admits an improved structural decomposition.  
\end{theorem}

\section{Equidistribution for twist KAM splittings}\label{sec:equi_KAM}

In this section, we introduce the notion of a (twist) KAM splitting and prove a generic equidistribution result for hypersurfaces admitting such a splitting.

\subsection{KAM splittings}
\label{subsec:KAM_splitting}

Let $M$ be a compact oriented $3$-manifold and let $\omega$ be a Hamiltonian structure on $M$. An \emph{invariant torus} is an embedded $2$-torus $T \subset M$ such that $\ker\omega$ is tangent to $T$. Therefore, the characteristic foliation of $\omega$ restricts to an oriented smooth foliation on $T$. We say that an invariant torus $T$ is \emph{linear} if the restriction of the characteristic foliation to $T$ is diffeomorphic to a linear foliation on the standard torus $T^2$. The foliation on a linear invariant torus $T$ naturally gives rise to a \emph{rotation direction} in the positive projectivization $\mathbb{P}_+H_1(T;\mathbb{R})$. After a choice of identification $T\simeq T^2  = \mathbb{C}/\left(\mathbb{Z}+\sqrt{-1}\mathbb{Z}\right)$, we can also view the rotation direction as a unit vector $z \in S^1 \subset \mathbb{C}$.

A vector $z = x + y\sqrt{-1}\in S^1$ is \emph{Diophantine} if there exist constants $c > 0$ and $\tau > 1$ such that
\begin{equation*}|k_1 x + k_2 y| \geq c|k|^{-\tau}\quad\text{for all}\quad k = (k_1, k_2) \in \mathbb{Z}^2\,\setminus\,\{0\}.\end{equation*}
The set of Diophantine vectors $z \in S^1$ has full Lebesgue measure.

A linear invariant torus is called a \emph{KAM torus} if its rotation direction, viewed as an element of $S^1 \subset \mathbb{C}$, is Diophantine. The property of being Diophantine is independent of the choice of identification $T\simeq T^2$.

Near a KAM torus $T$, the two-form $\omega$ admits a Birkhoff normal form.

\begin{prop}
\label{prop:Birkhoff_normal_form}
    Let $n>0$ be a positive integer. Then there exists an identification of a tubular neighborhood of $T$ with $T^2\times (-\varepsilon,\varepsilon)$ such that, with respect to the coordinates $(x,y,r)$ on $T^2\times (-\varepsilon,\varepsilon)$, the torus $T$ is given by $\{r=0\}$ and the $2$-form $\omega$ is given by
    \begin{equation}
    \label{eq:omega_birkhoff_normal_form}
        \omega = -b(r) dr \wedge dx + a(r) dr \wedge dy + O(r^{n+1})
    \end{equation}
    for polynomials $a(r)$ and $b(r)$ of degree at most $n$.
\end{prop}

The proof of this result is given in Subsection~\ref{ss:Birkhoff}.

Observe that the rotation direction of $T$ is parallel to $a(0) + b(0)\sqrt{-1}$. We say that the KAM torus $T$ is \emph{twist} if the complex numbers
\begin{equation*}
    a(0) + b(0)\sqrt{-1} \qquad \text{and} \qquad a'(0) + b'(0)\sqrt{-1}
\end{equation*}
are $\mathbb{R}$-linearly independent. In other words, $T$ is twist if the derivative at $r=0$ of the rotation direction of the family of linear foliations on $T^2$ specified by the family of $1$-forms $-b(r)dx + a(r)dy$ does not vanish.

\begin{Remark}
    The polynomials $a(r)$ and $b(r)$ in the above Birkhoff normal form are not uniquely determined. For instance, we may reparametrize $T^2$ by an element of $\operatorname{SL}(2,\mathbb{Z})$ and the interval $(-\varepsilon,\varepsilon)$ by a diffeomorphism fixing $0$. Whether or not $T$ is twist is independent of these choices.
\end{Remark}

Twist KAM tori are known to have remarkable robustness and stability properties. First, they persist under exact perturbations of $\omega$.

\begin{theorem}\label{thm:KAMrobust}
    Let $T \subset (M,\omega)$ be a twist KAM torus. Then, for every closed $2$-form $\tilde{\omega}$ cohomologous to $\omega$ and sufficiently $C^\infty$ close to $\omega$, there exists a twist KAM torus $\tilde{T}$ of $(M,\tilde{\omega})$ close to $T$ such that the rotation direction of $\tilde{T}$ agrees with the rotation direction of $T$ under the identification $\BP_+H_1(T;\BR) \cong \BP_+H_1(\tilde{T};\BR)$ induced by the isomorphisms $H_1(T;\BR) \cong H_1(N;\BR) \cong H_1(\tilde{T};\BR)$, where $N$ is a small tubular neighborhood of $T$ containing $\tilde{T}$. Moreover, the assignment $\tilde{\omega} \mapsto \tilde{T}$ is smooth and maps $\omega$ to $T$.
\end{theorem}
Secondly, they have a strong stability property: they are accumulated on both sides by other invariant tori. 
\begin{theorem}\label{thm:KAMstable}
    Suppose that $T \subset (M,\omega)$ is a twist KAM torus. Then $T$ is accumulated by twist KAM tori on both sides. More precisely, let $N$ be a tubular neighborhood of $T$ and let $N_\pm$ denote the two components of $N\setminus T$. Then there exist sequences $\left\{T_k^\pm\right\}$ of twist KAM tori contained in $N_\pm$, respectively, that converge to $T$ in the $C^\infty$ topology.
\end{theorem}
Theorems~\ref{thm:KAMrobust} and \ref{thm:KAMstable} are consequences of results of Moser and R\"ussmann on invariant circles of surface diffeomorphisms after passing to an annular local Poincar\'e section. We explain this in more detail in Section~\ref{sec:KAM_proofs}.

We can now formulate the key notion of a KAM splitting. 

\begin{defi}\label{def:KAMsplit}
A \emph{KAM splitting} of $(M,\omega)$ is a finite collection $\mathcal{T} = \{T_1, \ldots, T_r\}$ of KAM tori such that the following holds. Let $N \subset M$ be the closure of any component of $M \setminus\bigsqcup_{i=1}^r T_i$, compactifying each end with the appropriate $T_i$. Then $N$ satisfies at least one of the following two conditions:
\begin{itemize}
    \item [-] There exists a contact form $\alpha$ on $N$ such that $\omega|_N = d\alpha$;
    \item [-] The characteristic foliation of $\omega|_N$ admits a global cross section.  
\end{itemize}
A KAM splitting $\{T_1, \ldots, T_r\}$ is called \emph{twist} if each KAM torus $T_i$ is twist.
\end{defi}
\begin{Remark}\label{rem:KAMsplitting_endcompactification}
A priori, in a KAM splitting as defined above, the following slightly ill-behaved situation could occur: two boundary components of one of the domains $N$ are identified with a single torus in $\mathcal{T}$. One way to avoid this issue is to work with end compactifications of the connected components of $M\setminus \bigsqcup_{i=1}^rT_i$ instead of their closures. For simplicity of the upcoming proofs, we stick to closures and assume that we never have two boundary components of a domain $N$ corresponding to the same torus. This can be assumed without loss of generality for the following reason. Indeed, suppose there is a component $N$ with two boundary components corresponding to the same torus $T$. The invariant torus $T$ is accumulated by other KAM tori (by Theorem \ref{thm:KAMstable} above if $T$ is twist, otherwise by adapting the main result in \cite{fk09} to Hamiltonian structures, along the lines of what is done in Section \ref{sec:KAM_proofs}). By adding one of those accumulating tori, call it $T'$ (which is twist if $T$ was twist), to the KAM splitting, we add a new component to the complement of $\mathcal{T}$ of the form $V\cong T^2\times (a,b)$, where the flow can easily be seen to admit a global cross section (if the new torus is chosen close enough to $T$). That is, the former component $N$ is now divided into (the closure of) $N\setminus V$ and  $V$. This construction can be done on any torus where this issue happens, giving a new KAM splitting (which is twist if the original one was) of the Hamiltonian structure for which no component $N$ has two boundary components corresponding to the same invariant torus.
\end{Remark}

The goal of this section is to prove Theorem \ref{thm:KAMequi}, namely, that if $(M,\omega)$ admits a twist KAM splitting and $\omega$ is rational, then it satisfies the neighboring equidistribution property.

\subsection{Contact and mapping torus analogues}

Previous work by Irie \cite{Ir2} and Pirnapasov--Prasad \cite{PP} established analogs of Theorem~\ref{thm:KAMequi} for closed contact $3$-manifolds and rational area-preserving diffeomorphisms of compact surfaces with boundary. Our proof of Theorem~\ref{thm:KAMequi} makes essential use of these results. 

\paragraph{Near equidistribution in the contact case.} We start by stating the following equidistribution result of Irie \cite[Proposition 3.6]{Ir2}. 

\begin{prop}\label{prop:equidistribution_contact}
    Let $Y$ be a closed, oriented $3$-manifold and let $\lambda$ be any contact form on $Y$. Fix any $\varepsilon > 0$ and any finite collection $f_1, \ldots, f_k$ of smooth functions on $Y$. Then, there exists a $C^\infty$-small smooth function $h$ such that the contact form $\lambda' = e^h\lambda$ admits a characteristic current $c' \in \mathcal{C}_{\mathbb{R}}(d\lambda')$ satisfying
    \begin{equation} \label{eq:equidistribution_contact} \Big|\int_{c'} f_i\lambda' - \int_Y f_i\lambda' \wedge d\lambda'\Big| < \varepsilon \end{equation}
    for each $i = 1,\ldots,k$.
\end{prop}

\paragraph{Near equidistribution in the surface case.} Next, we provide the analogous statement for rational area-preserving diffeomorphisms of surfaces with boundary. The following statement can be deduced from~\cite{PP}.

\begin{prop}\label{prop:equidistribution_surface}
    Let $(\Sigma, \tau)$ be a compact symplectic surface, possibly with non-empty boundary. Let $\phi: \Sigma \to \Sigma$ be any rational area-preserving diffeomorphism. Fix any $\varepsilon > 0$ and any finite collection $f_1, \ldots, f_k$ of smooth functions on $\Sigma$. Then, there exists a $C^\infty$-small smooth Hamiltonian $H: S^1 \times \Sigma \to \mathbb{R}$ such that $H_t$ is locally constant on $\partial\Sigma$ for each $t \in S^1$, and $\phi' = \phi \circ \psi^1_H$ admits a current $\mathbf{x}' \in \mathcal{P}_{\mathbb{R}}(\phi')$ satisfying
    \begin{equation}\label{eq:equidistribution_surface}\Big|\mathbf{x}'(f_i) - \int_\Sigma f_i\tau \Big| < \varepsilon\end{equation}
    for each $i = 1,\ldots,k$. 
    
\end{prop}

\begin{proof}
    If $\partial\Sigma=\varnothing$, the proposition follows from \cite[Proposition 3.4]{PP} with $L=\varnothing$. If $\partial\Sigma\ne\varnothing$, it follows from the density of the set $\mathcal H_N$ established in the proof of Theorem~1.6 in Section~4.4 of~\cite{PP}.

    There is a small subtlety caused by the fact that the notion of rationality of area-preserving diffeomorphisms used in~\cite{PP} is slightly more restrictive than the notion of rationality we adopt in the present paper, see Remark~\ref{rem:different_notions_of_rationality}. However, this does not cause problems because what is actually needed in the proof of Theorem~1.6 in~\cite{PP} is an embedding of $\Sigma$ into a closed surface $\hat\Sigma$ and an extension of $\phi$ to a rational area-preserving diffeomorphism $\hat\phi$ of $\hat\Sigma$. The existence of such $\hat\Sigma$ and $\hat\phi$ is established in Lemma~\ref{lem:rational_extension_to_closed_surface}.
\end{proof}

It is most convenient for our purposes to work in mapping tori of surfaces.  The following proposition represents a translation of Proposition~\ref{prop:equidistribution_surface} to this setting. The statement involves the map $f_H$ from \eqref{eq:mt_bijection}. 

\begin{prop}\label{prop:equidistribution_mt}
    Let $(\Sigma, \tau)$ be a compact symplectic surface, possibly with non-empty boundary. Let $\phi: \Sigma \to \Sigma$ be any rational area-preserving diffeomorphism. Fix any $\varepsilon > 0$ and any finite collection $f_1, \ldots, f_k$ of smooth functions on $Y_\phi$. Then, there exists a $C^\infty$-small smooth Hamiltonian $H: Y_\phi \to \mathbb{R}$ such that $H$ is locally constant on the boundary of each fiber of $Y_\phi$, and $\phi' = \phi \circ \psi^1_H$ admits a characteristic current $c' \in \mathcal{C}_{\mathbb{R}}(\phi')$ satisfying
    \begin{equation}\label{eq:equidistribution_mt}\Big|\int_{f_H^{-1}(c')} f_i dt  - \int_{Y_{\phi}} f_i dt \wedge \tau_{\phi} \Big| < \varepsilon\end{equation}
    for each $i = 1,\ldots,k$. 
\end{prop}

\begin{proof}
For each $i$, define a smooth function $g_i: \Sigma \to \mathbb{R}$ by
$$x \mapsto \int_0^1 f_i(t, x) dt$$
where we view $f_i$ as a function on $[0,1] \times \Sigma$. It follows from Fubini's theorem that
\begin{equation}
\label{eq:fubini_consequence}
\int_\Sigma g_i \tau = \int_{Y_\phi} f_i dt \wedge \tau_{\phi}.
\end{equation}

Apply Proposition~\ref{prop:equidistribution_surface} to the functions $g_0\equiv 1, g_1, \ldots, g_k$. Then, there exists a $C^\infty$-small smooth Hamiltonian $H: [0,1] \times \Sigma \to \mathbb{R}$, such that $H_t$ is locally constant on $\partial\Sigma$ for each $t$, and $\mathbf{x} \in \mathcal{P}_{\mathbb{R}}(\phi^H)$ such that
\begin{equation}\label{eq:mapping1}
    \Big|\mathbf{x}(g_i) - \int_\Sigma g_i\tau \Big| < \frac{\varepsilon}{2}
\end{equation}
for each $i$. The function $g_0$ is added to have a uniform bound on the mass $\mathbf{x}(1)$ of $\mathbf{x}$, independent of the $C^\infty$ size of $H$. After applying a time reparameterization, we may assume that $H_t \equiv 0$ for $t$ near $0$ and $1$, and that $ H_t$ remains $C^\infty$-small. Therefore, we may view $H$ as a smooth function on $Y_\phi$ that is locally constant on the boundary of each fiber. 

Let $\phi' := \phi^H$ and let $c' \in \mathcal{C}_\mathbb{R}(\phi')$ be the characteristic current induced by $\mathbf{x}$. Let us attempt to relate $\mathbf{x}(g_i)$ and $\int_{f_H^{-1}(c')} f_i dt$. Since $H$ is $C^\infty$-small, we may assume that for an arbitrarily small $\delta$, which will be adjusted later in the proof, we have
\begin{equation}\label{eq:mapping2}
    \sup_{(t, x) \in [0,1] \times \Sigma} |f_i(t, x) - f_i(t, \psi_H^t(x))| < \delta
\end{equation}
for each $i$.

Write $\mathbf{x}=\sum_j a_j\mathbf{x}_j$. Then $c'=\sum_j a_jc_j$, where $c_j$ is the image of $\bigcup_{x\in\mathbf{x}_j}[0,1]\times\{x\}$ under the quotient map $[0,1]\times\Sigma\to Y_{\phi'}$. We expand
$$\mathbf{x}(g_i) = \sum_j\sum_{x \in \mathbf{x}_j} a_j g_i(x) = \sum_j\sum_{x \in \mathbf{x}_j} \int_0^1 a_jf_i(t, x) dt$$
and
\begin{multline*}
\int_{f_H^{-1}(c')} f_i dt = \int_{c'} (f_i \circ f_H^{-1}) dt \\ = \sum_j\sum_{x \in \mathbf{x}_j} a_j\int_{[0,1] \times \{x\}} (f_i \circ f_H^{-1}) dt = \sum_j\sum_{x \in \mathbf{x}_j} a_j\int_0^1 f_i(t, \psi^t_H(x)) dt.
\end{multline*}
Using the bound \eqref{eq:mapping2} it follows that
\begin{equation}\label{eq:mapping3}
|\mathbf{x}(g_i) - \int_{f_H^{-1}(c')} f_i dt| < \delta \cdot \mathbf{x}(1)
\end{equation}
for each $i$. We choose $\delta<\frac{\varepsilon}{2\mathbf{x}(1)}$, which is possible because $\mathbf{x}(1)$ is uniformly bounded. Using the triangle inequality, and considering the bounds \eqref{eq:mapping3} and \eqref{eq:mapping1} and the identity \eqref{eq:fubini_consequence}, we deduce that
\begin{align*}
    \Big|\int_{f_H^{-1}(c')} f_i dt - \int_{Y_\phi} f_i dt \wedge \tau_\phi \Big| &\leq \Big|\int_{f_H^{-1}(c')} f_i dt - \mathbf{x}(g_i)\Big| + \Big|\mathbf{x}(g_i) - \int_\Sigma g_i\tau\Big|\\ 
    &+ \Big|\int_\Sigma g_i\tau - \int_{Y_\phi} f_i dt \wedge \tau_\phi\Big| \\
    &< \varepsilon,    
\end{align*}
concluding the proof.
\end{proof}

\subsection{Near equidistribution for forms supported away from the splitting tori}

We will first prove the following proposition, which will be used to deduce an analogous near equidistribution result for arbitrary forms later in Theorem \ref{thm:nearEqRobSplit}.

\begin{prop}\label{prop:eqCompSupp}
    Let $\omega$ be a rational Hamiltonian structure on $M$. Suppose that $(M,\omega)$ admits a twist KAM splitting given by a finite collection of tori $T_1,...,T_r$. Let $N \subset M$ be the closure of a connected component of $M\setminus (T_1\cup \dots \cup T_r)$. Let $\kappa$ be a one-form on $M$ such that $\kappa \wedge \omega > 0$. Let $\alpha_1,...,\alpha_k$ be one-forms on $M$ with compact support in the interior of $N$. Then, for any $\varepsilon>0$, there exists an arbitrarily $C^\infty$-small cohomologous perturbation $\omega'\in \Omega^2(N)$ of $\omega|_N$ whose characteristic foliation is tangent to the boundary of $N$, and a current $c'\in \mathcal{C}_{\mathbb{R}}(\omega')$ such that
    \[
        \Big| \int_{c'} \alpha_i - \int_N \alpha_i \wedge \omega' \Big| <\varepsilon, \qquad i=1,\ldots,k,
    \]
    and
    \[
        \int_{c'} \kappa < \int_N \kappa\wedge \omega'+\varepsilon.
    \]
\end{prop}

\begin{proof}
    By definition of a twist KAM splitting, the region $N$ has contact type or admits a global cross section. We treat these two cases separately.
    
    \medskip
    \emph{The contact type case:} We assume that $N$ is of contact type. This means that $\omega|_N = d\lambda$ for some contact form $\lambda$. We may assume that $\lambda\wedge d\lambda>0$. Indeed, if $\lambda\wedge d\lambda<0$, apply the positive case on $-N$ with test forms $-\alpha_i$ and framing $-\kappa$, obtaining $\omega'$ and $c'_-$. Reversing every orbit orientation while retaining its positive weight gives a current $c'$ on $N$. For $\eta\in\{\alpha_1,\ldots,\alpha_k,\kappa\}$, we have $\int_{c'}\eta=\int_{c'_-}(-\eta)$ and $\int_N\eta\wedge\omega'=\int_{-N}(-\eta)\wedge\omega'$. Thus both required estimates hold on $N$, and the remaining conclusions are unchanged by reversing orientation.

    Our strategy is to apply Proposition~\ref{prop:equidistribution_contact}. In order to do so, we first extend $N$ to a closed contact manifold $\widetilde N$. This can be done as follows. Write $\xi := \ker(\lambda)$ for the contact structure on $N$ induced by $\lambda$. We glue a solid torus onto each connected component of the boundary of $N$. Let $\widetilde{N}$ denote the resulting closed manifold. The contact structure $\xi$ extends to a contact structure $\widetilde{\xi}$ on $\widetilde{N}$. This can be done, for example, by taking a contact structure on the solid torus such that the characteristic foliation on some concentric torus is isomorphic to that of $\xi$ along the corresponding boundary component, and appealing to the classical result of Giroux \cite[Section II, Proposition 1.2(b), p.~649]{Gir}. It follows that $\lambda$ also extends to a contact form $\widetilde{\lambda}$ on $\widetilde N$ such that $\ker(\widetilde{\lambda}) = \widetilde{\xi}$.

    Set $\widetilde\omega := d\widetilde\lambda$. Let $\widetilde{R}$ denote the Reeb vector field of $\widetilde{\lambda}$ and set $f_i := \alpha_i(\widetilde{R})$. Note that, by our assumptions on the $1$-forms $\alpha_i$, each function $f_i$ is compactly supported in $\operatorname{Int}(N) \subset \widetilde{N}$.
    
    Let $\hat \kappa$ be an arbitrary extension of $\kappa$ to a $1$-form on $\widetilde N$ such that $\hat \kappa \wedge \widetilde\omega> 0$. Using a suitable non-negative cutoff function, we obtain a $1$-form $\widetilde\kappa$ whose restriction to $N$ agrees with $\kappa$ and whose support is contained in a small neighborhood of $N$. Moreover, we have $\widetilde\kappa \wedge \widetilde\omega \geq 0$. Since $\widetilde\kappa$ was obtained from $\hat\kappa$ by multiplication by a non-negative function, we also have $\widetilde\kappa \wedge \eta \geq 0$ for any $2$-form $\eta$ sufficiently close to $\widetilde\omega$. We observe that we can choose the cutoff such that $\int_{\widetilde N} \widetilde \kappa\wedge \widetilde \omega$ is arbitrarily close to $\int_N \kappa \wedge \omega$.

    We define the function $f_0 :=\tilde\kappa (\widetilde R)$ and apply Proposition~\ref{prop:equidistribution_contact} to the contact form $\widetilde\lambda$ on $\widetilde N$ and the set of functions $f_0,\dots,f_k$. We deduce that there exists a $C^\infty$-small perturbation $\bar{\lambda}$ of $\widetilde{\lambda}$ and a characteristic current $\bar{c} \in \mathcal{C}_{\mathbb{R}}(d\bar{\lambda})$ such that
    \begin{equation}\label{eq:eqdstr_pf_2}\Big|\int_{\bar{c}} f_i\bar{\lambda} - \int_{\widetilde{N}} f_i\bar{\lambda} \wedge d\bar{\lambda}\Big| < \varepsilon\end{equation}
    for each $i$. Set $\overline\omega:=d\overline\lambda$.

    By Theorem~\ref{thm:KAMrobust}, the invariant boundary tori of $N$ are robust under perturbation. After conjugation by a $C^\infty$-small isotopy, we may therefore assume that the characteristic foliation of $\overline\omega$ is tangent to each of the original boundary tori of $N$, as opposed to small perturbations thereof. Estimate~\eqref{eq:eqdstr_pf_2} for $f_0$ gives a uniform mass bound on $\overline c$ near $N$, where the isotopy is supported. After decreasing the approximation tolerance in Proposition~\ref{prop:equidistribution_contact} if necessary, we may therefore assume that \eqref{eq:eqdstr_pf_2} continues to hold after conjugation.

    After the conjugation, each periodic orbit in the support of $\bar{c}$ is either completely contained in $N$ or does not intersect $N$ at all. Let $\omega'$ denote the restriction of $\overline\omega$ to $N$. Note that $\omega'$ is an exact perturbation of $\omega$ because both $\omega$ and $\omega'$ are exact on $N$.

    Let $c' \in \mathcal{C}_{\R}(\omega')$ be the intersection of $\overline c$ with $N$. It remains to check that $\omega'$ and $c'$ satisfy the desired estimates in the statement of the proposition.

    Since the functions $f_i$ for $i=1,\dots,k$ are compactly supported in $\operatorname{Int}(N)$, we deduce the identities
    \begin{equation}
    \label{eq:near_equidist_cpct_supp_proof_identity}
    \int_{\bar{c}} \alpha_i(\bar R) \bar{\lambda} = \int_{c'} \alpha_i,\qquad \int_{\widetilde{N}} \alpha_i(\bar R)\bar{\lambda} \wedge \overline\omega = \int_{N} \alpha_i \wedge \omega'
    \end{equation}
    for each $i = 1,\dots, k$. Here $\overline R$ denotes the Reeb vector field of $\overline \lambda$. On the other hand, we have
    \begin{equation}
    \label{eq:near_equidist_cpct_supp_proof_estimate_a}
    \Big|\int_{\bar c}\alpha_i(\bar R)\bar \lambda-\int_{\bar c}f_i \bar \lambda \Big| \leq \norm{\alpha_i(\bar R)-f_i} \int_{c'} \bar \lambda,
    \end{equation}
    and similarly
    \begin{equation}
    \label{eq:near_equidist_cpct_supp_proof_estimate_b}
    \Big|\int_{\widetilde{N}} \alpha_i(\bar R) \bar{\lambda} \wedge \overline\omega-\int_{\widetilde{N}}f_i\bar\lambda\wedge \overline\omega \Big|\leq \norm{\alpha_i(\bar R)-f_i} \int_{\widetilde{N}} \bar \lambda \wedge \overline\omega.
    \end{equation}
    In addition, estimate \eqref{eq:eqdstr_pf_2} for $f_0$ (which is positive on $N$) gives a uniform upper bound on the mass of $c'$ which is independent of how close we choose $\bar \lambda$ to $\tilde \lambda$. We can therefore make the right hand sides in estimates~\eqref{eq:near_equidist_cpct_supp_proof_estimate_a} and~\eqref{eq:near_equidist_cpct_supp_proof_estimate_b} arbitrarily small by choosing $\overline\lambda$ closer to $\widetilde\lambda$.

    Combining estimates~\eqref{eq:near_equidist_cpct_supp_proof_estimate_a} and~\eqref{eq:near_equidist_cpct_supp_proof_estimate_b}, identities~\eqref{eq:near_equidist_cpct_supp_proof_identity}, and estimate~\eqref{eq:eqdstr_pf_2} therefore yields the first of the two estimates in the statement of the proposition.
     
    An analogous argument shows that
    \begin{equation*}
        \Big| \int_{\overline{c}}\widetilde\kappa - \int_{\widetilde N} \widetilde\kappa \wedge \overline\omega \Big| < \varepsilon
    \end{equation*}
    Since $\widetilde\kappa \wedge \overline\omega\geq 0$, it follows that
    $$ \int_{c'}\kappa \leq \int_{\bar c} \tilde \kappa. $$
    Recall that $\tilde \kappa$ can be chosen such that $\int_{\widetilde N} \widetilde \kappa \wedge \widetilde\omega$ is arbitrarily close to $\int_N \kappa\wedge \omega$. Thus we can also arrange $\int_{\widetilde{N}}\widetilde\kappa \wedge \overline\omega$ to be arbitrarily close to $\int_N \kappa \wedge \omega'$. With the cutoff chosen sufficiently close to $N$, and after decreasing the earlier approximation tolerance and choosing the perturbation sufficiently small, we obtain the second required estimate. This concludes the proof of the proposition in the contact type case.

    \medskip
    \emph{The global cross section case:} Let us now assume that $N$ admits a global cross section $\Sigma \subset N$. The restriction $\tau := \omega|_\Sigma$ is an area form, and the first return map $\phi$ of the cross section $\Sigma$ is area-preserving with respect to $\tau$. The Hamiltonian structures $(N,\omega|_N)$ and $(Y_\phi,\tau_\phi)$ are diffeomorphic. An explicit diffeomorphism can be constructed as follows. Choose a nowhere vanishing vector field $R$ on $N$ which is positively tangent to the characteristic foliation and such that the first return time of the cross section $\Sigma$ is identically equal to $1$. Write $\phi^t$ for the flow of $R$. We define a diffeomorphism
    $$\Psi: Y_\phi \to N \qquad (t, p) \mapsto \phi^t(p).$$
    The pushforward of the vector field $\partial_t$ under this diffeomorphism is precisely $R$. Since the flow $\phi^t$ preserves $\omega$, it follows that the form $\Psi^*\omega$ restricts to $\tau$ on each surface fiber of $Y_\phi$, which further implies $\Psi^*\omega = \tau_\phi$. In the following, we identify $(N,\omega|_N)$ with $(Y_\phi,\tau_\phi)$ via this diffeomorphism.

    By assumption, the cohomology class $[\omega]$ is a real multiple of a rational class. Hence the same is true for the class $[\tau_\phi]$. This implies that $\phi$ is a rational area-preserving diffeomorphism.

    For $i = 1,\dots, k$ let us write the $1$-form $\alpha_i$ on $Y_\phi$ in the form $\alpha_i = f_i dt + \gamma_i$, where $f_i$ is a smooth function and $\gamma_i$ is a $1$-form such that $\gamma_i(R) = 0$. Similarly, we write $\kappa = f_0 dt + \gamma_0$.

    We apply Proposition~\ref{prop:equidistribution_mt} to the set of functions $f_0,\dots,f_k$. This yields a $C^\infty$-small Hamiltonian perturbation $\overline\phi = \phi \circ \psi_H^1$ of $\phi$ and a current $\overline c\in \mathcal{C}_\R(\overline \phi)$ such that
    \begin{equation}
    \label{eq:near_equidist_cpct_supp_proof_mapping_torus_case_estimate_a}
        \Big| \int_{f_H^{-1}(\overline c)} f_i dt - \int_{Y_\phi} f_i dt \wedge \tau_\phi \Big| < \varepsilon.
    \end{equation}
    We set
    \begin{equation*}
        \omega' := f_H^* \tau_{\overline\phi} = \tau_\phi + dH \wedge dt.
    \end{equation*}
    Clearly this is a $C^\infty$ small exact perturbation of $\tau_\phi$ whose characteristic foliation is tangent to the boundary of $Y_\phi$. Moreover, we set
    \begin{equation*}
        c' := f_H^{-1}(\overline{c}) \in \mathcal{C}_\R(\omega').
    \end{equation*}
    We need to check that $\omega'$ and $c'$ satisfy the desired estimates in the statement of the proposition.

    First, we observe that estimate~\eqref{eq:near_equidist_cpct_supp_proof_mapping_torus_case_estimate_a} for $i=0$ implies a uniform upper bound on the mass of $c'$ which is independent of the Hamiltonian perturbation of $\phi$. By choosing the Hamiltonian perturbation small, we can therefore make the integrals $\int_{c'}\gamma_i$ arbitrarily small. This implies that we can arrange $\int_{c'}\alpha_i$ to be arbitrarily close to $\int_{c'}f_i dt$ for $i = 1,\dots,k$ and $\int_{c'}\kappa$ to be arbitrarily close to $\int_{c'}f_0 dt$. We can also arrange $\int_{Y_\phi} f_i dt\wedge \tau_\phi$ to be arbitrarily close to $\int_{Y_\phi} \alpha_i \wedge \omega'$ and $\int_{Y_\phi} f_0dt\wedge \tau_\phi$ to be arbitrarily close to $\int_{Y_\phi} \kappa\wedge \omega'$. After decreasing the approximation tolerance in Proposition~\ref{prop:equidistribution_mt} and choosing the perturbation sufficiently small, we obtain the desired estimates.

\end{proof}

\subsection{Near equidistribution for twist KAM splittings}

We establish here a near-equidistribution theorem that will be the key step to prove Theorem \ref{thm:KAMequi}.

\begin{theorem}\label{thm:nearEqRobSplit}
       Let $M$ be a $3$-manifold with rational Hamiltonian structure $\omega$ that admits a twist KAM splitting $(M,\mathcal{T})$. For any $\varepsilon>0$, and any set of $1$-forms $\beta_1,...,\beta_k \in \Omega^1(M)$, there exists a $C^\infty$-small cohomologous perturbation $\omega'$ of $\omega$ and a current $c'\in \mathcal{C}_{\mathbb{R}}(\omega')$ such that
   $$  \Big|\int_{c'} \beta_i - \int_M \beta_i\wedge \omega' \Big|< \varepsilon, $$
   for each $i=1,\dots,k$.
\end{theorem}

\begin{proof}
     Let $\varepsilon>0$ and $\beta_1,\dots,\beta_k \in \Omega^1(M)$ be arbitrary. Without loss of generality, we assume that $\beta_1= \lambda$, a $1$-form such that $\lambda\wedge \omega>0$ and $\int_M \lambda \wedge \omega=1$. Let $\varphi:M\rightarrow [0,1]$ be a smooth cutoff function equal to $0$ in a small neighborhood $V$ of $\mathcal{T}$, and equal to $1$ outside a slightly larger neighborhood $V'\supset V$. Define forms $\alpha_i :=\varphi \beta_i$. Given an arbitrarily small $\varepsilon_1>0$ to be adjusted later in the proof, we can choose $V$ and $V'$ such that, for any $\beta_i$ and any $2$-form $\eta$ in a sufficiently small neighborhood of $\omega$, we have
    \begin{equation}\label{eq:cutoff}
        \Big| \int_M \beta_i \wedge \eta - \int_M \alpha_i \wedge \eta  \Big|<\varepsilon_1, \quad i=1,\dots,k.
    \end{equation}

    Let $N_1,\dots,N_m$ be the closures of the connected components of $M\setminus \mathcal{T}$. We let $\alpha_i^j$ denote the restriction of $\alpha_i$ to $N_j$. Note that $\alpha_i^j$ is compactly supported in the interior of $N_j$.

    We first fix the geometry for gluing local perturbations. By Theorem~\ref{thm:KAMstable}, for each boundary component $B$ of $N_j$ we can choose an inner parallel twist KAM torus $S_B$ of $\omega$, contained in $V$ and away from the support of $\varphi$. Let $N_j^0\subset\operatorname{Int}(N_j)$ be the domain bounded by these tori and containing the supports of the forms $\alpha_i^j$. Fix a smooth cutoff $\psi_j:M\to[0,1]$ supported in $\operatorname{Int}(N_j)$ and equal to $1$ on a neighborhood of $N_j^0$. The transition regions of $\psi_j$ lie between the tori $S_B$ and the corresponding boundary components $B$.

    By Theorem~\ref{thm:KAMrobust}, every sufficiently small cohomologous perturbation of $\omega|_{N_j}$ has nearby twist KAM tori bounding a domain $\widetilde N_j$ which contains the supports of all $\alpha_i^j$ in its interior and is contained in the region where $\psi_j=1$. Applying Proposition~\ref{prop:eqCompSupp} within this fixed neighborhood of $\omega|_{N_j}$ shows that for every $j = 1,\dots,m$ there is a $C^\infty$ small cohomologous perturbation $\omega_j\in \Omega^2(N_j)$ of $\omega|_{N_j}$ and a current $c_j\in\mathcal{C}_{\mathbb{R}}(\omega_j)$ such that
        \begin{equation}\label{eq:equidistNj}
            \Big|\int_{c_j} \alpha_i^j - \int_{N_j} \alpha_i^j\wedge \omega_j \Big|< \varepsilon_1, \quad i = 1,\dots, k
        \end{equation}  
    and
    \begin{equation}
    \label{eq:equidistNj_lambda}
        \int_{c_j} \lambda < \int_{N_j} \lambda\wedge \omega_j+\varepsilon_1.
    \end{equation}

    For each $j$, let $\widetilde N_j$ be the domain corresponding to the chosen perturbation $\omega_j$ in the preceding construction. Since $\omega_j$ and $\omega|_{N_j}$ are cohomologous, an auxiliary Hodge decomposition gives a primitive $\kappa_j$ of $\omega_j-\omega|_{N_j}$ which tends to zero in the $C^\infty$ topology as $\omega_j$ tends to $\omega|_{N_j}$. The cutoffs $\psi_j$ were fixed before choosing $\omega_j$. It follows that
    \[
        \omega'=\omega+\sum_{j=1}^m d(\psi_j\kappa_j)
    \]
    is an arbitrarily $C^\infty$-small cohomologous perturbation of $\omega$. By construction, $\omega'|_{\widetilde N_j}=\omega_j$.
 
    Since the characteristic foliation of $\omega_j$ is tangent to the boundary of $\widetilde N_j$, each closed characteristic appearing in $c_j$ with positive weight is either fully contained in $\widetilde N_j$ or does not intersect $\widetilde N_j$ at all. Let
    \begin{equation*}
     \tilde c_j=c_j|_{\widetilde N_j} \in \mathcal{C}_{\mathbb{R}}(\omega'|_{\widetilde N_j}) \subset \mathcal{C}_{\mathbb{R}}(\omega')
    \end{equation*}
    denote the restriction of $c_j$ to $\widetilde{N}_j$. Since the support of each $\alpha_i^j$ is contained in the interior of $\widetilde N_j$, estimate \eqref{eq:equidistNj} still holds with $\tilde c_j$ instead of $c_j$. Define $c'=\sum_{j=1}^m \tilde c_j$. Estimate~\eqref{eq:equidistNj} yields 
 \begin{equation}\label{eq:boundalpha}
     \Big|\int_{c'} \alpha_i - \int_M \alpha_i\wedge \omega' \Big|=\Big| \sum_{j=1}^m \Big(\int_{\tilde c_j} \alpha_i^j - \int_{N_j} \alpha_i^j \wedge \omega'\Big) \Big|< m\varepsilon_1. 
 \end{equation}
Since we can arrange $\omega'|_{N_j}$ to be arbitrarily $C^\infty$ close to $\omega_j$, we can in particular guarantee that
\begin{equation*}
    \Big|\int_{N_j} \lambda\wedge (\omega'-\omega_j)\Big| <\varepsilon_1.
\end{equation*}
Estimate~\eqref{eq:equidistNj_lambda} yields
\begin{equation}\label{eq:massbound}
     \int_{c'} \lambda=\sum_j \int_{\tilde c_j}\lambda \leq \sum_j \int_{c_j}\lambda<\sum_j\Big(\int_{N_j}\lambda\wedge \omega_j+ \varepsilon_1\Big) < \int_M\lambda\wedge \omega'+2m\varepsilon_1.
 \end{equation}

We claim that if $\varepsilon_1$ is chosen small enough, then
        \begin{equation}\label{eq:aprox}
            \Big|\int_{c'} \beta_i - \int_M \beta_i\wedge \omega' \Big|< \varepsilon,\enspace i=1,...,k,
        \end{equation}
    thus establishing the theorem. To obtain such an inequality, we will need a last estimate, which roughly says that the mass of $c'$ does not concentrate near $\mathcal{T}$.
    \begin{lemma}
       Let $\varepsilon_2>0$. If $\varepsilon_1$ is sufficiently small, then
    $$ \Big|\int_{c'} \beta_i - \int_{c'} \alpha_i \Big|< \varepsilon_2.$$ 
    \end{lemma}
    \begin{proof}
    Let $K>0$ be a constant such that $|\beta_i(R')| \leq K \lambda(R')$ for any vector field $R'$ positively tangent to the characteristic foliation of $\omega'$. Note that this constant $K$ can be chosen to be independent of the $C^\infty$ small perturbation $\omega'$.
    
    Recall that $\beta_1 = \lambda$ and $\alpha_1 = \varphi \lambda$. Estimate~\eqref{eq:boundalpha} therefore yields
    \begin{equation}
    \label{eq:auxiliary_estimate_lemma_integral_beta_alpha_over_c}
        \int_M \varphi\lambda \wedge \omega' -m\varepsilon_1 \leq 
        \int_{c'} \varphi \lambda.
    \end{equation}
    We estimate
    \begin{align*}
        \Big|\int_{c'} \beta_i - \int_{c'}\alpha_i \Big| &= \Big| \int_{c'}(1-\varphi)\beta_i \Big| \\
        &\leq K \int_{c'} (1-\varphi) \lambda \\
        &< K \Big( \int_M\lambda \wedge \omega' +2m\varepsilon_1 - \int_M\varphi\lambda \wedge\omega' +m\varepsilon_1 \Big) \\
        &= K\Big( 3m\varepsilon_1 + \int_M(\beta_1 - \alpha_1) \wedge \omega' \Big) \\
        &< K (3m+1) \varepsilon_1
    \end{align*}
    Here the first equality is immediate from the definition of $\alpha_i$. The first inequality follows from the inequality $|\beta_i(R')| \leq K\lambda(R')$. The second inequality uses estimates~\eqref{eq:massbound} and~\eqref{eq:auxiliary_estimate_lemma_integral_beta_alpha_over_c}. The second equality follows from the definition of $\alpha_1$. The final inequality is a consequence of estimate~\eqref{eq:cutoff}.

    Now simply take $\varepsilon_1$ small enough such that $K(3m+1)\varepsilon_1 < \varepsilon_2$. This concludes the proof of the lemma.

    \end{proof}

To conclude, we have
    \begin{align*}
        \Big|\int_{c'} \beta_i - \int_M \beta_i\wedge \omega' \Big|& \leq \Big|\int_{c'} \beta_i - \int_{c'} \alpha_i \Big| + \Big| \int_{c'}\alpha_i - \int_{M} \alpha_i \wedge \omega' \Big|\\
         & + \Big|\int_M \alpha_i \wedge \omega' - \int_M \beta_i\wedge \omega' \Big|\\
         &< \varepsilon_2 + m\varepsilon_1 + \varepsilon_1.
    \end{align*}
    Now simply choose $\varepsilon_1,\varepsilon_2$ small enough such that $\varepsilon_2 + m\varepsilon_1 + \varepsilon_1 < \varepsilon$. This concludes the proof of the theorem.
\end{proof}

\subsection{Proof of Theorem \ref{thm:KAMequi}}

Let $M \subset (W,\Omega)$ be a hypersurface such that the induced Hamiltonian structure $\omega$ on $M$ is rational and admits a twist KAM splitting $(M,\mathcal{T})$. Let us identify a tubular neighborhood of $M$ in $W$ with $\R\times M$. Given a hypersurface $N$ sufficiently close to $M$, the natural projection $\pi: \R\times M \rightarrow M$ restricts to a diffeomorphism $\pi|_N : N \rightarrow M$. Pushing forward the Hamiltonian structure on $N$ induced by $\Omega$ via this diffeomorphism yields a Hamiltonian structure $\omega_N$ on $M$ which is an exact perturbation of $\omega$.

By Theorem \ref{thm:KAMrobust}, there is a neighborhood $\mathcal{U}\subset \mathcal{HS}(W)$ of $M$ such that, for every hypersurface $N$ in $\mathcal{U}$, the Hamiltonian structure $\omega_N$ admits a twist KAM splitting $\mathcal{T}_N$ close to $\mathcal{T}$.

Let $\{\beta_i\}_{i\in \mathbb{N}}$ be a countable family of $1$-forms in $\Omega^1(M)$ which is dense. We define a subset $\mathcal{U}_{k,\varepsilon} \subset \mathcal{U}$ as follows. A hypersurface $N$ is in $\mathcal{U}_{k,\varepsilon}$ if there is a current $c_N\in \mathcal{C}_{\mathbb{R}}(\omega_N)$ which is a linear combination of non-degenerate closed characteristics of $\omega_N$ such that 
$$ \Big| \int_{c_N} \beta_i - \int_M \beta_i \wedge \omega_N\Big|<\varepsilon, \text{ for } i=1,...,k.$$
Here a closed characteristic is called non-degenerate if the linearization of the first return map has no eigenvalue equal to $1$.

We claim that $\mathcal{U}_{k,\varepsilon}$ is open and dense. For $N\in \mathcal{U}_{k,\varepsilon}$, the closed characteristics showing up in $c_N$ with positive weights are non-degenerate by definition and therefore robust. This implies openness of $\mathcal{U}_{k,\varepsilon}$.

 Let us prove that $\mathcal{U}_{k,\varepsilon}$ is dense. Let $N$ be any hypersurface in $\mathcal{U}$. An application of Theorem \ref{thm:nearEqRobSplit} shows that there is a $C^\infty$-small cohomologous perturbation $\omega'$ of $\omega_N$ and a current $c' \in \mathcal{C}_{\mathbb{R}}(\omega')$ such that 
\begin{equation}\label{eq:nearpert}
    \Big| \int_{c'} \beta_i - \int_M \beta_i \wedge \omega'\Big|<\varepsilon, \text{ for } i=1,...,k.  
\end{equation}
After a further arbitrarily small cohomologous perturbation, we may in addition assume that the closed characteristics of $c'$ are non-degenerate, see e.g. \cite[Lemma 19]{Rob}. 

Finally, an application of \cite[Lemma 49]{C} shows that there is an embedding $e: N\rightarrow W$ which is $C^\infty$ close to the inclusion of $N$ such that $e^*\Omega = (\pi|_N)^*\omega'$. The hypersurface $N' := e(N)$ is a $C^\infty$ approximation of $N$ which belongs to $\mathcal{U}_{k,\varepsilon}$.

To justify this claim, note that membership in $\mathcal U_{k,\varepsilon}$ requires \eqref{eq:nearpert} for the current and Hamiltonian structure obtained by transporting $c'$ and $\omega'$ to $N'$ via $e\circ(\pi|_N)^{-1}$ and then projecting back to $M$ via $\pi|_{N'}$. The resulting Hamiltonian structure is $C^\infty$-close to $\omega'$, and the resulting current is obtained from $c'$ by $C^\infty$-small deformations of its constituent closed characteristics, with their weights unchanged. To control the resulting change in \eqref{eq:nearpert}, we include a fixed framing among the test forms when applying Theorem~\ref{thm:nearEqRobSplit}, obtaining a uniform mass bound on $c'$, as in the proofs of Propositions~\ref{prop:equidistribution_mt} and~\ref{prop:eqCompSupp}. This bound controls the change in the current integrals, while the volume integrals vary continuously. By initially applying Theorem~\ref{thm:nearEqRobSplit} with a smaller tolerance, we therefore ensure that \eqref{eq:nearpert} still holds after these perturbations of the current and Hamiltonian structure. This concludes the proof of density.

Let $\varepsilon_k$ be a sequence of numbers such that $\varepsilon_k \rightarrow 0$. Then the set
$$\mathcal{V}= \bigcap_{k} \mathcal{U}_{k,\varepsilon_k}$$
is residual in $\mathcal{U}$. One easily checks that every hypersurface in $\mathcal{V}$ admits equidistributed closed characteristics.

\begin{Remark}\label{rem:rational}
For the proofs of Theorem \ref{thm:nearEqRobSplit} and hence of Theorem \ref{thm:KAMequi}, the rationality assumption on the Hamiltonian structure $\omega$ can be slightly weakened. It is enough that $\omega$ restricts to a rational Hamiltonian structure on each connected component of $M\setminus\bigcup_{T\in\mathcal T}T$ individually. Note that components of contact type are exact and therefore automatically rational. Hence it suffices to assume rationality for all components which are not of contact type and therefore admit a global cross section. This will be useful in the proof of Theorem \ref{thm:main2}.
\end{Remark}

\subsection{Some remarks}

We finish this section with some remarks about KAM splittings.
\medskip

\textbf{On the neighboring $C^\infty$-closing property.} First, let us comment that in order to deduce the neighboring $C^\infty$-closing property for a hypersurface $M$, it is enough to assume the following property, which is a priori weaker than the existence of a twist KAM splitting:

There should exist a $C^\infty$-neighborhood $\mathcal{V}$ of $M$ in $\mathcal{HS}(W)$, and a dense subset $\mathcal{U}\subset \mathcal{V}$ satisfying the following. Choose any $N \in \mathcal{U}$ with induced Hamiltonian structure $\omega_N$. Any point $p\in N$ belongs to the interior of some connected domain $A \subset N$ whose boundary is a finite collection of tori tangent to the characteristic foliation of $N$, and such that $(A, \omega_N)$ is either of contact type or $\omega_N|_A$ is rational and its characteristic foliation admits a global cross section.
\medskip

\textbf{Hypersurfaces without KAM splittings.} It is not clear whether admitting a KAM splitting is a common (say, dense or generic) property among hypersurfaces in $\mathbb{R}^4$ or other symplectic manifolds. It is certainly not a $C^k$-dense property with $k\geq 2$ in full generality, as one can find very concrete examples of closed hypersurfaces in (open) symplectic manifolds that $C^k$ robustly have no KAM splittings. This follows from the discussion in \cite[Section 3.1]{CG1}, where it is explained that there exist certain smooth volume-preserving Anosov flows, even on rational homology spheres, that are not Reeb-like or orbit equivalent to suspensions of Anosov diffeomorphisms. Such Anosov flows are of a particular kind known as ``non $\mathbb{R}$-covered". Recall that Anosov flows admit no linear invariant tori. Hence those examples do not admit a KAM splitting. This property persists under $C^1$-perturbations of the vector field, due to the structural stability of Anosov flows. The symplectization of a framed Hamiltonian structure whose Reeb field is parallel to this Anosov vector field yields an example of an open symplectic manifold and a closed hypersurface for which no hypersurface in a $C^2$-neighborhood admits a KAM splitting. Note, however, that Anosov flows admit an equidistributed sequence of periodic orbits anyway.

One can also construct examples of hypersurfaces with no KAM splitting by a local deformation of any given hypersurface, inserting a Reeb cylinder in its characteristic foliation using the symplectic plug construction in \cite[Section 2]{GG}; see also \cite{ReTh}. Such a cylinder cannot intersect a KAM torus and obstructs both contact type and the existence of a global cross section by the Stokes argument in \cite[Remark 2.6]{GG}. However, admitting such a cylinder is not a robust property, even in the $C^\infty$-topology.

\section{Equidistributed orbits near stable hypersurfaces}\label{sec:equi_stable}
In this section, we will use Theorem \ref{thm:KAMequi} to prove Theorems \ref{thm:main} and \ref{thm:main2}.

\subsection{KAM splittings of stable hypersurfaces}

We give a sufficient condition for a stable hypersurface to admit a twist KAM splitting. We will show in Subsection~\ref{ss:appstable} below that, for any stable hypersurface, this condition is satisfied after a $C^\infty$-small perturbation. Our sufficient condition is stated in terms of an improved structural decomposition with an additional non-constantness assumption on the slope of the integrable regions.

\begin{theorem}\label{thm:StRobust}
Let $M$ be a closed three-manifold with a stable Hamiltonian structure $(\lambda,\omega)$. Assume that it admits an improved structural decomposition $(N_0, N_c, N_{int})$ for which the kernel of $\omega$ has a non-constant slope on each connected component of $N_{int}$. Then $(M,\omega)$ admits a twist KAM splitting.
\end{theorem}

\begin{proof}

Consider a component $U \cong T^2 \times [0,1]$ of $N_{int}$. With respect to coordinates $(x,y,r)$ on $T^2\times [0,1]$, the Hamiltonian structure is given by
\begin{equation}\label{eq:wT2}
\omega=k_1(r)dr\wedge dx + k_2(r)dr\wedge dy.
\end{equation}
Recall the slope $k_U$ given by
\begin{equation*}
    k_U = \frac{-k_2 + k_1\sqrt{-1}}{|-k_2+k_1\sqrt{-1}|}.
\end{equation*}
By assumption, this is a non-constant function. An invariant torus in $U$ of the form $T^2\times \{r\}$ is a twist KAM torus precisely if $k_U(r)$ is Diophantine and $k'_U(r)\neq 0$.

Recall that the intersection of $N_0\cup N_c$ and $U$ is the union of $T^2\times [0,a]$ and $T^2\times [b,1]$ for some $0<a<b<1$. Let $C_0$ and $C_1$ be the connected components of $N_0 \cup N_c$ containing $T^2\times \{0\}$ and $T^2\times \{1\}$, respectively.

\begin{lemma}\label{lem:finitetori}
There exist finitely many values $0\leq r_1<r_2<...<r_m\leq 1$ with $m\geq 2$ satisfying the following properties:
\begin{enumerate}
\item the image of the slope function restricted to each interval $[r_{i},r_{i+1}]$ with $i=1,...,m-1$ is contained in the interior of a semicircle of $S^1$;
\item on each torus $T^2\times \{r_i\}$, the slope $k_U(r_i)$ is Diophantine and $k'_U(r_i) \neq 0$, i.e.\ $T^2\times \{r_i\}$ is a twist KAM torus;
\item after possibly enlarging $N_0$ and $N_c$, we can choose $r_1, r_m$ such that $T^2\times \{r_1\} \subset C_0$ and $T^2\times \{r_m\}\subset C_1$.
\end{enumerate}
\end{lemma}
\begin{proof}
The first two statements follow immediately from the fact that the slope function is non-constant and that Diophantine vectors are dense in $S^1$. The only part that needs justification is the third item. Let us focus on $r_1$, since the argument for $r_m$ is identical. Let $B_0=T^2\times [0,a]$ be the intersection between $C_0$ and $U$. If the restriction of the slope function $k_U|_{B_0}$ is non-constant, then it is straightforward that one can choose $r_1\in [0,a]$. Hence, we shall assume that the slope is constant on $[0,a]$, and we will enlarge $C_0$ to a domain $C_0'$ such that the slope is non-constant on $C_0'\cap U$.\\

\textbf{Case 1:} We assume that $C_0\subset N_c$. In this case, by Lemma \ref{lem:lambda_contact_region} and Item 5 in the definition of improved structural decomposition (Definition \ref{def:struc}), the form $\lambda$ is a contact form and $\omega$ is a constant multiple of $d\lambda$. It follows that $\omega$ is of contact type on $C_0$ and hence also on $T^2\times [0,a]$. 

Let $\alpha$ be a contact primitive of $\omega|_{C_0}$ (one can take a constant multiple of $\lambda$). As in \cite[Lemma 3.9]{CV}, let $\hat \alpha$ be the one-form obtained by averaging $\alpha$ under the natural $T^2$-action on $T^2\times [0,a]$. The one-form $\hat \alpha$ is a contact form and it satisfies $d\hat \alpha=\omega$. Moreover, it satisfies $\hat \alpha=\alpha + dG$ for some function $G$. In order to see this, note that $\hat \alpha - \alpha$ is a closed form with vanishing average and that averaging of a closed form yields a cohomologous form.

If $\varphi(r)$ is a cut-off function taking values in $[0,1]$, equal to $0$ near $r=0$ and equal to $1$ near $r=\delta$ for some $\delta \ll 1$, the one-form $\tilde \alpha= \alpha + d(\varphi(r)G)$, extended as $\alpha$ on $C_0\setminus U$, is a primitive of $\omega$ on $C_0$. Since $dr\wedge\omega=0$, the form $\tilde\alpha\wedge\omega$ is a convex combination of $\alpha\wedge\omega$ and $\hat\alpha\wedge\omega$, which have the same nonzero sign. Thus $\tilde\alpha$ is contact on $T^2\times[0,a]$. We claim that the one-form $\tilde \alpha$ extends naturally as a $T^2$-invariant primitive of $\omega$ across $T^2\times [a,1]$. Indeed, $\tilde \alpha$ is given by
\begin{equation}\label{eq:avprim}
\tilde \alpha= h_1(r)dx+ h_2(r)dy+ h_3(r)dr,
\end{equation}
on $T^2\times [\delta,a]$, where 
\begin{equation}\label{eq:der}
(h_1'(r),h_2'(r))=(k_1(r),k_2(r)).
\end{equation}
Extend $h_3$ arbitrarily for $r\in [a,1]$, and extend the functions $h_1,h_2$ for $r\in [a,1]$ by integrating Equation \eqref{eq:der}. Then the expression \eqref{eq:avprim} extends to $T^2\times [a,1]$ as a primitive $\tilde \alpha$ of $\omega$ on all of $C_0\cup U$. This primitive is a contact form for $r\in [0,a]$, and the contact condition reads
$$ h_2(r)h_1'(r)-h_1(r)h_2'(r)\neq 0, \enspace \text{for } r\in [\delta,a]. $$
 Define 
$$ v:=\operatorname{max}_{r\in [0,1]} \{ r \mid k_U|_{[0,r]}\equiv k_U(0)\},$$
so that $[0,v]$ is the maximal interval containing $0$ such that $k_U(r)$ is constant. It clearly contains $[0,a]$ by assumption, and $v<1$. We claim that $\tilde \alpha \wedge d\tilde \alpha\neq 0$ on $T^2\times[0,d]$ for some $d\in(v,1)$. Since $\tilde\alpha$ is already a contact form on $T^2\times[0,a]$, it suffices to show that $W(r)=h_2(r)h_1'(r)-h_1(r)h_2'(r)$ does not vanish on $[a,v]$; extension slightly beyond $v$ then follows by continuity. Assume that there is $r^*\in [a,v]$ such that $W(r^*)=0$. The fact that the slope is constant can be written as the condition $h_2'(r)=Ah_1'(r)$ for some constant $A$ (we assume for simplicity that $h_1'(r)\neq 0$, the other case being analogous), and hence $h_2(r)=Ah_1(r)+B$ for some other constant $B$ on $[a,v]$. Thus
\begin{align*}
   W(r)&= h_2(r)h_1'(r)-h_2'(r)h_1(r)\\
    &=(Ah_1(r)+B)h_1'- Ah_1'(r)h_1(r)\\
    &= Bh_1'(r).
\end{align*}
The contact condition at $r=a$ gives $W(a)\neq0$, and hence $B\neq0$. We deduce that $h_1'(r^*)=0$. This is a contradiction, since then $h_2'(r^*)=0$ and $\omega$ would vanish along the torus $r=r^*$. We conclude that $\omega$ admits a contact primitive on $C_0\cup \left(T^2\times [a,d]\right)$ for some $d$ slightly larger than $v$. Redefining the contact region $N_c$ by changing $C_0$ into $C_0'=C_0\cup \left(T^2\times [a,d]\right)$, the slope on $C_0'\cap U$ is no longer constant, and hence we can choose $r_1$ such that $T^2\times \{r_1\}\subset C_0'\cap U$ where $T^2\times \{r_1\}$ has a Diophantine rotation vector.\\

\textbf{Case 2:} We assume that $C_0\subset N_0$. From Equation \eqref{eq:wT2}, it follows that the kernel of $\omega$ is spanned by 
$$X=k_2(r)\partial_x-k_1(r)\partial_y. $$
On $C_0$ choose a closed one-form $\tilde\lambda$ such that $\tilde\lambda\wedge\omega$ is nowhere zero and $\tilde\lambda(X)>0$ on $B_0$. Up to perturbing $\tilde \lambda$ within closed one-forms, we can assume that it defines a rational cohomology class. We can average $\tilde \lambda$ on $T^2\times [0,a]$, obtaining a $T^2$-invariant closed one-form $\lambda'$ that is positive on $X$ and cohomologous to $\tilde \lambda$. By \cite[Lemma 3.1]{CV2}, there is a closed one-form $\hat \lambda$ positive on $X$ that is equal to $\tilde \lambda$ near $T^2\times \{0\}$ and equal to $\lambda'$ for $r>\delta$ for some small $\delta$, keeping fixed the cohomology class. On $T^2\times [\delta,a]$, we can write $\hat \lambda$ as 
$$ \hat \lambda= Adx+Bdy+c(r)dr, $$
for some constants $A,B$ and a function $c$. To simplify, we can further assume that near $r=a$ the function $c$ vanishes, simply by considering $Adx+Bdy+\rho(r)c(r)dr$ for a cut-off function $\rho$ equal to $0$ near $r=a$. 
Let $v$ be as before the maximal value in  $[0,1]$ such that the slope of $\ker \omega$ is constant in $[0,v]$. Extend $\hat \lambda$ beyond $r=a$ to $T^2\times[a,1]$ by $A\,dx+B\,dy$. By construction,
\begin{equation}\label{eq:N0pos}
\hat \lambda(X)=k_2(r)A-k_1(r)B>0
\end{equation}
for $r$ near $a$. On $[0,v]$, the vector $X(r)$ is a positive scalar multiple of $X(a)$ because its oriented slope is constant. Thus the extension $A\,dx+B\,dy$ is positive on $X$ for $r\in[a,v]$, and by continuity for $r\in[a,d]$ with some $d\in(v,1)$. Since the original patched form is positive on $X$ on $T^2\times[0,a]$, the extended form $\hat\lambda$ is positive on $X$ on $T^2\times J$, where $J=[0,d]$. The slope is non-constant on $J$ by the definition of $v$. Extending the patched form by $\tilde\lambda$ on $C_0\setminus U$ yields a closed one-form on $C_0\cup\left(T^2\times J\right)$ that is nowhere zero on the characteristic direction. Arguing as in Tischler's theorem, since $[\hat \lambda]$ is a multiple of an integer cohomology class, the form $\hat \lambda$ is a multiple of the pullback of the standard form on $S^1$ by a fiber bundle $\pi:C_0\cup \left(T^2\times J\right)\rightarrow S^1$. Any fiber of $\pi$ is a cross section of $\ker\omega$ there, since $\hat\lambda$ is nonzero on the characteristic direction. We can thus redefine $N_0$ by replacing $C_0$ by $C_0'=C_0\cup \left(T^2\times J\right)$, and choose $r_1$ such that $T^2\times \{r_1\}\subset C_0'$.
\end{proof}

We apply Lemma \ref{lem:finitetori} to each connected component of the integrable region of $M$. This yields a finite family of twist KAM tori $T_1,...,T_s$. Let $V$ be the closure of a connected component of $M\setminus \bigsqcup_{i=1}^sT_i$. Then $V$ satisfies one of the following properties:
\begin{itemize}
\item[-] it is contained in $N_c$,
\item[-] it is contained in $N_0$,
\item[-] it is diffeomorphic to $T^2\times I$, the flow is linear on each torus fiber, and the image of the slope is contained in the interior of a semicircle.
\end{itemize}
In the first case, the domain $V$ is of contact type. In the second case, it is a surface bundle over the circle and any fiber is a cross section. In the last case, since the image of the slope function is contained in the interior of a semicircle, there exists a closed one-form $\eta_V$ on $V$ that evaluates positively on the characteristic foliation. Up to perturbation, it can be assumed to have a rational cohomology class, and hence it defines a surface bundle over the circle whose fibers are transverse to the characteristic foliation of $V$. Hence, the closure of any connected component of $M\setminus \bigsqcup_{i=1}^s T_i$ satisfies the required properties for $\mathcal{T}=\{T_1,...,T_s\}$ to define a twist KAM splitting of $(M,\omega)$.
\end{proof} 

\subsection{Density of stable hypersurfaces with the neighboring equidistribution property} \label{ss:appstable}

We proceed to apply the previous results to establish the main theorems of this work.
A straightforward consequence of Theorem \ref{thm:KAMequi} and Theorem \ref{thm:StRobust} is the following.
\begin{corollary}\label{cor:eqSt}
Let $M\subset (W,\Omega)$ be a closed stable hypersurface in a symplectic $4$-manifold such that the induced Hamiltonian structure $\omega$ on $M$ is rational. If there is a stabilizing one-form $\lambda$ of $\omega$ such that $(\lambda,\omega)$ admits an improved structural decomposition $(N_0, N_c, N_{int})$ for which the slope function is non-constant on each connected component of $N_{int}$, then $M$ satisfies the neighboring equidistribution property.
\end{corollary}

To deduce Theorem \ref{thm:main}, we need to show that an arbitrary rational stable hypersurface $M$ can be perturbed into a stable hypersurface satisfying the hypotheses of Corollary \ref{cor:eqSt}. Even though it is not needed for our results, it is worth mentioning that the perturbation can be realized by a ``stable isotopy" \cite[Section 6.6]{CV}. This means that the perturbation is given by an isotopy of embedded submanifolds $M_t$ with $t\in [0,1]$ such that there exists a smooth family of one-forms $\lambda_t\in \Omega^1(M_t)$ stabilizing the induced family of Hamiltonian structures on $M_t$.

\begin{lemma}\label{lem:slope}
Let $M\subset (W,\Omega)$ be a stable hypersurface. Then there exists a stable hypersurface $M'$ which is $C^\infty$ close to $M$ and isotopic to $M$ via a stable isotopy such that the induced Hamiltonian structure $\omega'$ on $M'$ admits a stabilizing $1$-form $\lambda'$ with the property that $(\lambda', \omega')$ admits an improved structural decomposition for which the slope function is non-constant on each component of the integrable region.
\end{lemma}
\begin{proof}
By Theorem~\ref{thm:struc} the stabilizable Hamiltonian structure $\omega=\Omega|_M$ admits a stabilizing $1$-form $\lambda$ for which there exists an improved structural decomposition $(N_0,N_c,N_{int})$.

Given any connected component $U\cong T^2\times [0,1]$ of the integrable region, by \cite[Lemma 6.2]{CR} there exist arbitrarily $C^\infty$-small one-forms $\gamma, \eta$ compactly supported in $U$ such that $(\lambda+\gamma, \omega+d\eta)$ is a stable Hamiltonian structure integrable on $U$ and whose slope is non-constant there. Strictly speaking, \cite[Lemma 6.2]{CR} asserts compact support of the perturbation of $\omega$, but not of a primitive. To ensure the latter, we additionally require $\int_0^1(\widetilde h_2-h_2)\,dt=0$ in the construction of the cited proof, using its notation.

An inspection of the proof of \cite[Lemma 6.2]{CR} shows that this works parametrically, i.e.\ there exists a $C^\infty$-small family 
 $$(\gamma_t,\eta_t) \qquad \text{with} \qquad t\in [0,1],$$ 
 satisfying $\gamma_0=\eta_0\equiv 0$ and $\gamma_1=\gamma$, $\eta_1=\eta$ and such that $(\lambda+\gamma_t,\omega+d\eta_t)$ is a path of cohomologous stable Hamiltonian structures. Then \cite[Lemma 49]{C} and its proof, which is an application of Moser's path method, show that there exists a $C^\infty$-small family of embeddings $\varphi_t: M \longrightarrow W$ such that $\varphi_t^*\Omega=\omega+d\eta_t$.
 
 Apply this to each connected component of $N_{int}$ and observe that the resulting stable Hamiltonian structure admits the same improved structural decomposition, now with non-constant slope functions.
\end{proof}

\subsection{Equidistribution for geodesible conservative flows}
\label{ss:app3D}

In this section we prove Theorem~\ref{thm:main2}, which is an application to 3D conservative dynamics. Let $M$ be a closed three-dimensional manifold endowed with a volume form $\mu$. Recall that $\mathfrak{X}_{\mu}(M)$ denotes the set of nowhere-vanishing vector fields that preserve $\mu$ and that $\mathcal{SR}_\mu(M)$ denotes the subset of vector fields parallel to the Reeb field of a stable Hamiltonian structure on $M$. Equivalently, $\mathcal{SR}_\mu(M)$ can be characterized as the subset of vector fields $X \in \mathfrak{X}_\mu(M)$ such that $\iota_X\mu$ is stabilizable, or as the subset of vector fields that are geodesible, see \cite{Re}.

Theorem \ref{thm:main2} states that a vector field in $\mathcal{SR}_\mu(M)$ can be $C^\infty$ perturbed to a vector field in $\mathfrak{X}_\mu(M)$ that admits an equidistributed sequence of periodic orbits. This does not immediately follow from Theorem \ref{thm:main}. Indeed, the stabilizable Hamiltonian structure $\iota_X\mu$ associated to a vector field $X\in \mathcal{SR}_\mu(M)$ might not be rational. A priori, it is not clear whether $\iota_X\mu$ can be perturbed to a rational Hamiltonian structure while maintaining the property that it is stabilizable. We proceed differently, showing that rationality can be achieved exactly where it is needed in our proof, see Remark \ref{rem:rational}.

\begin{proof}[Proof of Theorem \ref{thm:main2}]
Let $X$ be a vector field in $\mathcal{SR}_\mu(M)$. Let $\omega := \iota_X\mu$ be the induced stabilizable Hamiltonian structure. By Theorem~\ref{thm:struc} we may choose a stabilizing $1$-form $\lambda$ for $\omega$ such that $(\lambda,\omega)$ admits an improved structural decomposition.

Consider the symplectization $M\times (-\varepsilon,\varepsilon)$ equipped with the symplectic form $\Omega := \omega + d(t\lambda)$, where $t$ is the coordinate of the second factor. The Hamiltonian structure induced on the hypersurface $M \cong M\times\{0\}$ is precisely given by $\omega$. It follows from Lemma~\ref{lem:slope} that there exists a $C^\infty$ perturbation $\tilde M$ of $M$ such that the induced Hamiltonian structure $\tilde \omega$ on $\tilde M$ admits a stabilizing $1$-form $\tilde\lambda$ such that $(\tilde\lambda,\tilde\omega)$ admits an improved structural decomposition satisfying the hypotheses of Theorem~\ref{thm:StRobust}. There are two cases.

\medskip
\noindent \textbf{Case 1.} If $\omega$ (and hence $\tilde\omega$) is rational, then Theorems~\ref{thm:StRobust} and~\ref{thm:KAMequi} imply that we can further $C^\infty$ perturb $\tilde M$ to a hypersurface $\hat M$ whose induced Hamiltonian structure $\hat \omega$ admits an equidistributed sequence of closed characteristics. The restriction of the natural projection $\pi: M\times (-\varepsilon,\varepsilon) \rightarrow M$ to $\hat{M}$ is a diffeomorphism $\pi|_{\hat M}:\hat M\to M$. We push forward $\hat\omega$ via this diffeomorphism and define $Y$ to be the vector field on $M$ characterized by $\iota_Y\mu = (\pi|_{\hat M})_*\hat\omega$. This yields the desired volume-preserving vector field $C^\infty$ close to $X$ with an equidistributed sequence of closed characteristics.

\medskip
\noindent \textbf{Case 2.} If $\omega$ (and hence $\tilde \omega$) is irrational, we proceed as follows. Via the diffeomorphism $\pi|_{\tilde M} : \tilde M \to M$, we may regard $\tilde \omega$ as a stabilizable Hamiltonian structure on $M$ which is $C^\infty$ close to $\omega$. By Theorem~\ref{thm:StRobust} it admits a KAM splitting. Let $N_1,\dots, N_r$ be the closures of the connected components of $M\setminus \left( T_1\cup \dots\cup T_k\right)$, where $T_1,\dots,T_k$ are the tori of the KAM splitting. Let $N$ be one of the domains $N_1,\dots,N_r$ on which $\tilde \omega$ admits a global cross section. This means that there exist a compact symplectic surface with boundary $(\Sigma,\tau)$, a symplectomorphism $\phi : (\Sigma,\tau)\to (\Sigma,\tau)$, and a diffeomorphism $\Psi : (Y_\phi,\tau_\phi) \to (N,\tilde\omega)$.

Using \cite[Proposition 3.3]{PP} exactly as in the proof of \cite[Theorem 1.5]{PP}, we find a rational symplectomorphism $\hat\phi : (\Sigma,\tau) \to (\Sigma, \tau)$ which is $C^\infty$ close to $\phi$ and agrees with $\phi$ near the boundary of $\Sigma$. We remove the mapping torus $(Y_\phi,\tau_\phi)$ from $(M,\tilde\omega)$ and replace it by $(Y_{\hat\phi},\tau_{\hat\phi})$. We can regard the result as a Hamiltonian structure on $M$ which is $C^\infty$ close to $\tilde \omega$ and agrees with $\tilde\omega$ outside of $N$. We repeat this step for every domain $N_i$ with a global cross section. Let $\hat\omega$ denote the resulting Hamiltonian structure on $M$. It is still $C^\infty$ close to $\tilde \omega$, admits the same KAM splitting $T_1,\dots,T_k$, and has the property that it restricts to a rational Hamiltonian structure on each domain $N_i$ with a global cross section.

Now choose a framing one-form $\eta$ for $\hat\omega$, consider the symplectic form $\Omega= \hat \omega + d(t\eta)$ on $M\times (-\varepsilon, \varepsilon)$ and apply Theorem \ref{thm:KAMequi} to the hypersurface $M\cong M\times \{0\}$, which is possible by Remark \ref{rem:rational}. This gives a hypersurface $M'$ which is $C^\infty$ close to $M$ and has an equidistributed sequence of closed characteristics. Let $\omega'$ be the pushforward to $M$ of the induced Hamiltonian structure on $M'$ via the restriction of the projection $M\times (-\varepsilon,\varepsilon)\to M$. Let $Y$ be the vector field on $M$ characterized by $\iota_Y\mu = \omega'$. It is $C^\infty$ close to $X$ and has an equidistributed sequence of periodic orbits, as desired.
\end{proof}

\section{Twist KAM tori of Hamiltonian structures}\label{sec:KAM_proofs}

In this section, we prove Proposition~\ref{prop:Birkhoff_normal_form}, i.e.\ we construct a Birkhoff normal form for Hamiltonian structures near KAM tori. This normal form was used in the definition of twist KAM tori in Subsection~\ref{subsec:KAM_splitting}. We also prove Theorems~\ref{thm:KAMrobust} and~\ref{thm:KAMstable} on the robustness and stability of twist KAM tori. We will deduce these results from theorems about invariant circles of surface maps, which we recall first.

\subsection{Review of invariant circles}

Let $(\Sigma,\tau)$ be a surface with an area form. Let $U \subset \Sigma$ be an open subset, and let $\varphi : (U,\tau) \hookrightarrow (\Sigma,\tau)$ be an area-preserving embedding. An \emph{invariant circle} of $\varphi$ is an embedded circle $C \subset U$ such that $\varphi(C) = C$. Choose an orientation of $C$. If $\varphi|_C$ preserves orientation, it has a well-defined rotation number $\alpha\in\mathbb{R}/\mathbb{Z}$. We say that $C$ is \emph{Diophantine} if $\varphi|_C$ preserves orientation and its rotation number is Diophantine. If this is the case, then $\varphi|_C$ is smoothly conjugate to a rigid circle rotation of rotation number $\alpha$; see~\cite{H8,Y2}. There is a Birkhoff normal form theorem near Diophantine invariant circles; see \cite{B2}, \cite[\S 1.3]{yoc92}, \cite[Proposition 5]{fk09}.

\begin{theorem}
\label{thm:birkhoff_normal_form_invariant_circle}
    Suppose that $C$ is a Diophantine invariant circle of $\varphi$ of rotation number $\alpha$. Let $n>0$ be a positive integer. Then there exists an identification of a tubular neighborhood of $C$ with $S^1 \times (-\varepsilon,\varepsilon)$ such that, with respect to the coordinates $(\theta,r)$ on $S^1 \times (-\varepsilon,\varepsilon)$, the circle $C$ is given by $\left\{ r = 0\right\}$, the area form $\tau$ is given by $d\theta \wedge dr$, and $\varphi$ is given by
    \begin{equation}
    \label{eq:birkhoff_normal_form_invariant_circle}
        \varphi(\theta,r) = \left(\theta+\alpha +\sum_{i=1}^{n} b_i r^i,\enspace r\right) + O(r^{n+1})
    \end{equation}
    near $C$. The numbers $b_i$ for $i \geq 1$ are uniquely determined. 
\end{theorem}

The numbers $b_i$ in the above theorem are called \emph{Birkhoff invariants}. We say that a Diophantine invariant circle is \emph{twist} if the first Birkhoff invariant $b_1$ does not vanish.

Suppose now that $\Sigma$ is an annulus and that $C$ is an essential circle. We say that an area-preserving embedding $\varphi : (U,\tau)\hookrightarrow (\Sigma,\tau)$ defined in a neighborhood $U$ of $C$ and sending $C$ to a circle $\varphi(C)$ homologous to $C$ is \emph{exact} if $\int_Z \tau$ vanishes for one (and hence any) $2$-chain $Z$ in $\Sigma$ with boundary $\partial Z = \varphi(C) -C$. The following result is a consequence of R\"ussmann's translated curve theorem~\cite{R1}; see~\cite[\S 1.4, Rem. 3]{yoc92}. 

\begin{theorem}
    \label{thm:perturbation_of_twist_diophantine_invariant_circles}
    Suppose that $C$ is a twist Diophantine invariant circle of $\varphi$ of rotation number $\alpha$. Then for every exact $C^\infty$ small perturbation $\tilde{\varphi}$ of $\varphi$ there exists a twist Diophantine invariant circle $\tilde{C}$ of $\tilde{\varphi}$ of the same rotation number $\alpha$. Moreover, the assignment $\tilde{\varphi} \mapsto \tilde{C}$ is smooth and maps $\varphi$ to $C$.
\end{theorem}

The next result follows from Moser's invariant curve theorem~\cite{Mo1}, see also \cite[Thm. 1]{fk09} for a more general result.

\begin{theorem}
    \label{thm:accumulation_twist_diophantine_invariant_circles}
    Suppose that $C$ is a twist Diophantine invariant circle of $\varphi$. Then $C$ is accumulated on both sides by twist Diophantine invariant circles of $\varphi$.
\end{theorem}

\subsection{Birkhoff normal form and the twist condition}\label{ss:Birkhoff}

In this subsection, we prove Proposition~\ref{prop:Birkhoff_normal_form}. Let $(M, \omega)$ be a Hamiltonian $3$-manifold, and let $T \subset M$ be an invariant KAM torus. The goal is to construct a Birkhoff normal form for $\omega$ near $T$.

Since the characteristic foliation on $T$ is diffeomorphic to a linear foliation, it admits a global cross section $C\subset T$ which is an embedded circle. Choose a thin embedded annulus $\Sigma \subset M$ which intersects the torus $T$ transversely in the circle $C$ and is moreover transverse to the characteristic foliation on $M$. The annulus $\Sigma$ serves as a local Poincar\'e section for the characteristic foliation on $M$, with a local return map $\varphi : U \hookrightarrow \Sigma$ defined on some open neighborhood $U\subset \Sigma$ of the circle $C$. The Hamiltonian structure $\omega$ restricts to an area form $\tau$ on $\Sigma$, and the local return map $\varphi$ is area-preserving with respect to $\tau$.

The circle $C$ is an invariant circle of $\varphi$, and since $T$ is a KAM torus, its rotation number $\alpha$ is Diophantine. We may therefore apply Theorem~\ref{thm:birkhoff_normal_form_invariant_circle}. For a given $n\geq 1$, this yields local coordinates $(\theta,r) \in \mathbb{T}\times (-\varepsilon,\varepsilon)$ on $\Sigma$ in a neighborhood of $C$ such that $C \simeq \mathbb{T}\times \{0\}$, the area form is given by $\tau = d\theta \wedge dr$, and $\varphi$ satisfies identity~\eqref{eq:birkhoff_normal_form_invariant_circle}. Let us write this identity as
\begin{equation*}
    \varphi(\theta,r) = (\theta + f(r),r) + O(r^{n+1}),
\end{equation*}
where $f(r) = \alpha + \sum_{i=1}^n b_i r^i$. Let $\varphi_0$ be the area-preserving diffeomorphism given by
\begin{equation*}
    \varphi_0(\theta,r) = (\theta +f(r),r).
\end{equation*}
Let $F(r)$ be the primitive of $f(r)$ which vanishes at $r = 0$. Viewing $F$ as a Hamiltonian on the annulus, we have $\varphi_0 = \psi_{F}^1$.

\begin{lemma}
There exists a $1$-periodic Hamiltonian $H$ of the form
\begin{equation*}
    H(t,\theta,r) = F(r) + O(r^{n+2})
\end{equation*}
such that $\varphi = \psi_H^1$.
\end{lemma}

\begin{proof}
Set $\eta=\varphi_0^{-1}\circ\varphi$.
Then
\[
    \eta(\theta,r)=(\theta,r)+O(r^{n+1}).
\]
Near $\mathbb{T}\times \{0\}$, the diffeomorphism $\eta$ admits a generating function $W$ such that $(\Theta,R) = \eta(\theta,r)$ is characterized by
\[
    \Theta-\theta=\partial_2W(\Theta,r),
    \qquad
    R-r=-\partial_1W(\Theta,r).
\]
We normalize $W$ such that $W|_{\mathbb{T}\times \{0\}} = 0$. From
\[
    \partial_2W(\Theta,r)=\Theta-\theta=O(r^{n+1})
\]
we deduce that
\[
    W(\Theta,r)=O(r^{n+2}).
\]

For $t\in[0,1]$, let
\[
    W_t=tW,
\]
and let $\eta_t$ denote the corresponding isotopy of area-preserving
diffeomorphisms connecting $\operatorname{id}$ to $\eta$. Let $K_t$ be its generating Hamiltonian, normalized to vanish on $\mathbb{T}\times \{0\}$. The Hamilton--Jacobi equation gives
\[
    \partial_tW_t(\Theta,r)
    =
    K_t\bigl(\Theta,r-\partial_1W_t(\Theta,r)\bigr).
\]
Since $\partial_tW_t=W=O(r^{n+2})$ and
$\partial_1W_t=O(r^{n+2})$, it follows that
\[
    K_t(\theta,r)=O(r^{n+2}).
\]

Finally, consider the Hamiltonian isotopy $\psi_{F}^t\circ\eta_t$ with time-one map $\varphi_0\circ\eta=\varphi$. It is generated by the Hamiltonian
\[
    H(t,\theta,r)
    =
    F(r)+K_t\circ(\psi_{F}^t)^{-1}(\theta,r)
    =
    F(r)+K_t\bigl(\theta-tf(r),r\bigr).
\]
Consequently,
\[
    H(t,\theta,r)=F(r)+O(r^{n+2}).
\]
The Hamiltonian $H$ might not extend to be $1$-periodic in time, but this can be easily fixed by reparametrizing the isotopy $\eta_t$ such that it is constant near $t=0,1$.
\end{proof}

Now consider $T^2 \times (-\varepsilon,\varepsilon)$ with coordinates $(t,\theta,r)$ and equipped with the Hamiltonian structure
\begin{equation*}
    \omega' = d\theta \wedge dr + dH \wedge dt = d\theta \wedge dr - f(r) dt \wedge dr + O(r^{n+1}).
\end{equation*}
Note that, up to relabeling the coordinates, this $2$-form exactly has the desired normal form of identity~\eqref{eq:omega_birkhoff_normal_form}.

We observe that $T^2\times \{0\}$ is an invariant torus of the Hamiltonian structure $\omega'$. The annulus $\{0\} \times \mathbb{T} \times (-\varepsilon,\varepsilon)$ is a local Poincar\'e section of the characteristic foliation, with local first return map given by $\varphi = \psi_H^1$ near the invariant circle $\{0\} \times \mathbb{T} \times \{0\}$. This implies that a neighborhood of $T$ in $(M,\omega)$ is diffeomorphic to a neighborhood of $T^2\times \{0\}$ in $(T^2\times (-\varepsilon,\varepsilon),\omega')$. An explicit diffeomorphism can be constructed by first choosing vector fields on $(M,\omega)$ and $(T^2\times (-\varepsilon,\varepsilon),\omega')$ positively tangent to the respective characteristic foliations such that the local first return times of the respective local annular Poincar\'e sections $\Sigma$ and $\{0\}\times \T\times (-\varepsilon,\varepsilon)$ are identically equal to $1$. Recall that we have identified a neighborhood of $C$ in $\Sigma$ with $\T\times (-\varepsilon,\varepsilon)$. Using the flows of the chosen vector fields, this identification can be extended to a local diffeomorphism between the Hamiltonian structures $\omega$ and $\omega'$ near the respective invariant tori. We refer to \cite[Proposition 3.8]{Ed2014} for a more detailed exposition of an analogous argument.

We have thus constructed the desired Birkhoff normal form of $\omega$ near the invariant KAM torus $T$. This finishes the proof of Proposition~\ref{prop:Birkhoff_normal_form}.\qed

\begin{Remark}
\label{rem:twist_tori_and_circles}
We conclude this subsection by pointing out that $T$ is a twist KAM torus if and only if $C$ is a twist Diophantine invariant circle of $\varphi$. Indeed, by definition, the twist condition on $T$ says that
\begin{equation*}
    f(0) +\sqrt{-1} \qquad \text{and} \qquad f'(0)
\end{equation*}
are $\mathbb{R}$-linearly independent complex numbers. This is the case if and only if $f'(0)$ does not vanish. But $f'(0)$ is precisely the first Birkhoff invariant $b_1$ of the invariant circle $C$ of $\varphi$, so this is further equivalent to $C$ being a twist Diophantine invariant circle of $\varphi$.
\end{Remark}

\subsection{Robustness and stability of invariant tori}

The goal of this subsection is to deduce Theorems~\ref{thm:KAMrobust} and~\ref{thm:KAMstable} on the robustness and stability of twist KAM tori of Hamiltonian structures from Theorems~\ref{thm:perturbation_of_twist_diophantine_invariant_circles} and~\ref{thm:accumulation_twist_diophantine_invariant_circles} on the robustness and stability of twist Diophantine invariant circles.

We begin with the proof of Theorem~\ref{thm:KAMrobust}, whose statement we recall for convenience.

\begin{theorem*}[Theorem \ref{thm:KAMrobust}]

    Let $T \subset (M,\omega)$ be a twist KAM torus. Then, for every closed $2$-form $\tilde{\omega}$ cohomologous to $\omega$ and sufficiently $C^\infty$ close to $\omega$, there exists a twist KAM torus $\tilde{T}$ of $(M,\tilde{\omega})$ close to $T$ such that the rotation direction of $\tilde{T}$ agrees with the rotation direction of $T$ under the identification $\BP_+H_1(T;\BR) \cong \BP_+H_1(\tilde{T};\BR)$ induced by the isomorphisms $H_1(T;\BR) \cong H_1(N;\BR) \cong H_1(\tilde{T};\BR)$, where $N$ is a small tubular neighborhood of $T$ containing $\tilde{T}$. Moreover, the assignment $\tilde{\omega} \mapsto \tilde{T}$ is smooth and maps $\omega$ to $T$.
\end{theorem*}

\begin{proof}
    Let us fix an annular local Poincar\'e section $\Sigma$ intersecting $T$ transversely in a circle $C$. Let $\varphi : (U,\tau) \hookrightarrow (\Sigma,\tau)$ be the local first return map defined on some neighborhood $U$ of $C$ in $\Sigma$. Since $T$ is a twist Diophantine invariant torus, $C$ is a twist Diophantine invariant circle of $\varphi$, see Remark~\ref{rem:twist_tori_and_circles}.

    Consider a $C^\infty$ small cohomologous perturbation $\tilde\omega$ of $\omega$. We may reduce ourselves to the case that $\tilde\omega$ induces the same area form $\tau$ on $\Sigma$ as $\omega$. In order to see this, let $\tilde\tau$ be the area form on $\Sigma$ induced by $\tilde\omega$. Using Moser's path method, we can construct an area-preserving embedding $(U,\tilde\tau) \hookrightarrow (\Sigma,\tau)$ which is $C^\infty$ close to the inclusion of $U \subset \Sigma$. We extend this embedding to a diffeomorphism $\psi$ of $M$ which is $C^\infty$ close to the identity. Now simply replace $\tilde\omega$ by the pushforward $\psi_*\tilde\omega$. After possibly shrinking the local Poincar\'e section $\Sigma$, the perturbation $\tilde\omega$ then induces the area form $\tau$ on $\Sigma$.

    After the above reduction, the perturbation $\tilde{\omega}$ gives rise to a perturbed return map $\tilde{\varphi} : (U,\tau) \hookrightarrow (\Sigma,\tau)$. The assumption that $\tilde{\omega}$ is cohomologous to $\omega$ implies that $\tilde{\varphi}$ is an exact perturbation of $\varphi$. Indeed, let $Z$ be a $2$-chain in $\Sigma$ such that $\partial Z = \tilde\varphi(C) - C$. Let $A \subset M$ be the annulus tangent to the characteristic foliation of $\tilde\omega$ with boundary $\partial A = C - \tilde\varphi(C)$. Then $A + Z$ forms a $2$-cycle in $M$ which is homologous to $T$. We compute
    \begin{equation*}
        0 = \int_T \omega = \int_{Z+A} \tilde\omega = \int_Z \tau.
    \end{equation*}
    Here the first equality uses that $T$ is tangent to the characteristic foliation of $\omega$. The second equality uses that $\omega$ and $\tilde\omega$ are cohomologous and that $T$ and $Z+A$ are homologous. The last equality follows because $A$ is tangent to the characteristic foliation of $\tilde\omega$. But the identity $0 = \int_Z \tau$ precisely means that $\tilde\varphi$ is exact.

     We are therefore in a position to apply Theorem~\ref{thm:perturbation_of_twist_diophantine_invariant_circles}. We deduce that there exists a twist Diophantine invariant circle $\tilde{C}$ of $\tilde{\varphi}$ such that $\tilde{\varphi}|_{\tilde{C}}$ has the same rotation number as $\varphi|_C$ and such that $\tilde{C}$ depends smoothly on $\tilde{\varphi}$ and hence on $\tilde{\omega}$. The union of all characteristic leaves of $\tilde{\omega}$ intersecting $\tilde{C}$ forms the desired twist Diophantine invariant torus $\tilde{T}$ of $\tilde{\omega}$.
\end{proof}

Next, we turn to the proof of Theorem~\ref{thm:KAMstable}.

\begin{theorem*}[Theorem \ref{thm:KAMstable}]
    Suppose that $T \subset (M,\omega)$ is a twist KAM torus. Then $T$ is accumulated by twist KAM tori on both sides. More precisely, let $N$ be a tubular neighborhood of $T$ and let $N_\pm$ denote the two components of $N\setminus T$. Then there exist sequences $\left\{T_k^\pm\right\}$ of twist KAM tori contained in $N_\pm$, respectively, that converge to $T$ in the $C^\infty$ topology.
\end{theorem*}

\begin{proof}
    Let $\Sigma$ be an annular local Poincar\'e section intersecting $T$ transversely in a circle $C$, and let $\varphi : (U,\tau) \hookrightarrow (\Sigma,\tau)$ be the local first return map. The circle $C$ is a twist Diophantine invariant circle of $\varphi$, see Subsection~\ref{ss:Birkhoff}. By Theorem~\ref{thm:accumulation_twist_diophantine_invariant_circles}, the circle $C$ is accumulated on both sides by twist Diophantine invariant circles. We can pick sequences $\left\{C_k^\pm\right\}$ of such invariant circles converging to $C$ from both sides. Define $T_k^\pm$ to be the union of all characteristics intersecting $C_k^\pm$. This defines sequences of twist Diophantine invariant tori converging to $T$ from both sides.
\end{proof}

\bibliographystyle{amsplain}
\bibliography{biblio}

@article{AGZ22,
  author  = {Albers, Peter and Geiges, Hansj{\"o}rg and Zehmisch, Kai},
  title   = {Pseudorotations of the 2-disc and {Reeb} flows on the 3-sphere},
  journal = {Ergodic Theory and Dynamical Systems},
  volume  = {42},
  number  = {2},
  year    = {2022},
  pages   = {402--436},
  doi     = {10.1017/etds.2021.15}
}

@article{FlHr25,
  author  = {Florio, Anna and Hryniewicz, Umberto},
  title   = {Quantitative conditions for right-handedness of flows},
  journal = {Annales Scientifiques de l'\'Ecole Normale Sup\'erieure},
  volume  = {58},
  number  = {4},
  year    = {2025},
  pages   = {899--941},
  doi     = {10.24033/asens.2620}
}

@article{HT,
  author={Hutchings, Michael and Taubes, Clifford Henry},
  title={{The Weinstein conjecture for stable Hamiltonian structures}},
  journal={Geometry \& Topology},
  volume={13},
  number={2},
  pages={901--941},
  year={2009}
}

@article{GG,
  author={Ginzburg, Viktor L. and G{\"u}rel, Ba{\c s}ak Z.},
  title={{Fragility and persistence of leafwise intersections}},
  journal={Mathematische Zeitschrift},
  volume={280},
  number={3-4},
  pages={989--1004},
  year={2015},
  doi={10.1007/s00209-015-1459-y}
}

@article{C,
  title={Stability is not open or generic in symplectic four-manifolds},
  author={Cardona, Robert},
  journal={arXiv preprint arXiv:2305.13158},
  year={2023}
}

@article{CR,
  title={{Periodic orbits and Birkhoff sections of stable Hamiltonian structures}},
  author={Cardona, Robert and Rechtman, Ana},
  journal={Journal de l’{\'E}cole polytechnique—Math{\'e}matiques},
  volume={12},
  pages={235-286},
  year={2025}
}

@article{CV,
  title={{First steps in stable Hamiltonian topology}},
  author={Cieliebak, Kai and Volkov, Evgeny},
  journal={Journal of the European Mathematical Society},
  volume={17},
  number={2},
  pages={321--404},
  year={2015}
}

@article{CV2,
  title={{A note on the stationary Euler equations of hydrodynamics}},
  author={Cieliebak, Kai and Volkov, Evgeny},
  journal={Ergodic Theory and Dynamical Systems},
  volume={37},
  number={2},
  pages={454--480},
  year={2017},
  publisher={Cambridge University Press}
}

@article{CGPZ,
  title={{Periodic Floer homology and the smooth closing lemma for area-preserving surface diffeomorphisms}},
  author={Cristofaro-Gardiner, Dan and Prasad, Rohil and Zhang, Boyu},
  journal={arXiv preprint arXiv:2110.02925},
  year={2021}
}

@article{CGPPZ,
  title={{A note on U-cyclic elements in monopole Floer homology}},
  author={Cristofaro-Gardiner, Dan and Pomerleano, Daniel and Prasad, Rohil and Zhang, Boyu},
  journal={Proceedings of the American Mathematical Society, Series B},
  volume={12},
  pages={218--228},
  year={2025},
  doi={10.1090/bproc/146}
}

@article{EH,
  title={{PFH spectral invariants and $C^\infty$ closing lemmas}},
  author={Edtmair, Oliver and Hutchings, Michael},
  journal={arXiv preprint arXiv:2110.02463},
  year={2021}
}

@article{Ed2014,
  title={Disk-like surfaces of section and symplectic capacities},
  author={Edtmair, Oliver},
  journal={Geometric and Functional Analysis},
  volume={34},
  number={5},
  pages={1399--1459},
  year={2024},
  publisher={Springer}
}

@article{FH,
  title={Almost existence from the feral perspective and some questions},
  author={Fish, Joel W. and Hofer, Helmut H. W.},
  journal={Ergodic Theory and Dynamical Systems},
  volume={42},
  number={2},
  pages={792--834},
  year={2022},
  publisher={Cambridge University Press}
}

@article{Gir,
  title={Convexit{\'e} en topologie de contact},
  author={Giroux, Emmanuel},
  journal={Commentarii Mathematici Helvetici},
  volume={66},
  pages={637--677},
  year={1991},
  publisher={Springer}
}

@article{Ir,
  title={Dense existence of periodic {R}eeb orbits and {ECH} spectral invariants},
  author={Irie, Kei},
  journal={Journal of Modern Dynamics},
  volume={9},
  number={1},
  pages={357--363},
  year={2015},
  publisher={Journal of Modern Dynamics}
}

@article{Ir2,
  title={Equidistributed periodic orbits of ${C}^\infty$-generic three-dimensional {R}eeb flows},
  author={Irie, Kei},
  journal={Journal of Symplectic Geometry},
  volume={19},
  number={3},
  pages={531--566},
  year={2021},
  publisher={International Press of Boston}
}

@article{PP,
  title={Generic equidistribution for area-preserving diffeomorphisms of compact surfaces with boundary},
  author={Pirnapasov, Abror and Prasad, Rohil},
  journal={Revista Matem{\'a}tica Iberoamericana},
  volume={41},
  number={3},
  pages={1101--1128},
  year={2025}
}

@article{Re,
  title={Existence of periodic orbits for geodesible vector fields on closed 3-manifolds},
  author={Rechtman, Ana},
  journal={Ergodic Theory and Dynamical Systems},
  volume={30},
  number={6},
  pages={1817--1841},
  year={2010},
  publisher={Cambridge University Press}
}

@article{Rob,
  title={{Generic properties of conservative systems II}},
  author={Robinson, R Clark},
  journal={American Journal of Mathematics},
  volume={92},
  number={4},
  pages={897--906},
  year={1970},
  publisher={JSTOR}
}

@article{P,
  title={Generic equidistribution of periodic orbits for area-preserving surface maps},
  author={Prasad, Rohil},
  journal={International Mathematics Research Notices},
  volume={2024},
  number={24},
  pages={14802--14834},
  year={2024},
  doi={10.1093/imrn/rnac340},
  publisher={Oxford University Press}
}

@article{Y,
  title={{Travaux de Herman sur les tores invariants}},
  author={Yoccoz, Jean-Christophe},
  journal={S{\'e}minaire Bourbaki},
  volume={1991},
  pages={92},
  year={1992}
}

@article{CG1,
  title={Nondensity results in high-dimensional stable {Hamiltonian} topology},
  author={Cardona, Robert and Gironella, Fabio},
  journal={Journal of the London Mathematical Society},
  volume={111},
  number={4},
  pages={e70143},
  year={2025},
  publisher={Wiley Online Library}
}

@article{AI,
  title={A {$C^\infty$} closing lemma for {H}amiltonian diffeomorphisms of closed surfaces},
  author={Asaoka, Masayuki and Irie, Kei},
  journal={Geometric and Functional Analysis},
  volume={26},
  number={5},
  pages={1245--1254},
  year={2016},
  publisher={Springer}
}

@book{HZ,
  title={Symplectic invariants and {H}amiltonian dynamics},
  author={Hofer, Helmut and Zehnder, Eduard},
  year={2012},
  publisher={Birkh{\"a}user}
}

@phdthesis{ReTh,
  title={Use and disuse of plugs in foliations},
  author={Rechtman, Ana},
  year={2009},
  school={Ecole normale sup{\'e}rieure de Lyon - ENS Lyon}
}

@article{Herm1,
  title={{Exemples de flots hamiltoniens dont aucune perturbation en topologie $C^{\infty}$ n'a d'orbites p{\'e}riodiques sur un ouvert de surfaces d'{\'e}nergies}},
  author={Herman, M-R},
  journal={CR Acad. Sci. Paris S{\'e}r. I Math.},
  volume={312},
  pages={989--994},
  year={1991}
}

@article{Smale,
  title={Mathematical problems for the next century},
  author={Smale, Steve},
  journal={The mathematical intelligencer},
  volume={20},
  number={2},
  pages={7--15},
  year={1998},
  publisher={Springer}
}

@article{Herm2,
  title={{Diff{\'e}rentiabilit{\'e} optimale et contre-exemples {\`a} la fermeture en topologie $C^\infty$ des orbites r{\'e}currentes de flots hamiltoniens}},
  author={Herman, Michael-R},
  journal={Comptes rendus de l'Acad{\'e}mie des sciences. S{\'e}rie 1, Math{\'e}matique},
  volume={313},
  number={1},
  pages={49--51},
  year={1991}
}

@article{Pr,
  title={Periodic points of rational area-preserving homeomorphisms},
  author={Prasad, Rohil},
  journal={Ergodic Theory and Dynamical Systems},
  volume={45},
  number={9},
  pages={2890--2907},
  year={2025},
  doi={10.1017/etds.2025.10}
}

@article{GLP,
  title={Area preserving homeomorphisms of surfaces with rational rotational direction},
  author={Guih{\'e}neuf, Pierre-Antoine and Le Calvez, Patrice and Passeggi, Alejandro},
  journal={Annales Henri Lebesgue},
  volume={8},
  pages={329--372},
  year={2025},
  doi={10.5802/ahl.237}
}

@article{B2,
  author  = {Birkhoff, George D.},
  title   = {Note sur la stabilit\'e en Dynamique},
  journal = {J. Math. Pures Appl.},
  volume  = {15},
  year    = {1936},
  pages   = {339--344}
}

@article{fk09,
  author  = {Fayad, Bassam and Krikorian, Rapha\"el},
  title   = {Herman's last geometric theorem},
  journal = {Ann. Sci. \'Ecole Norm. Sup. (4)},
  volume  = {42},
  number  = {2},
  year    = {2009},
  pages   = {193--219}
}

@article{H8,
  author  = {Herman, Michael R.},
  title   = {Sur la conjugaison diff\'erentiable des diff\'eomorphismes du cercle \`a des rotations},
  journal = {Publ. Math. Inst. Hautes \'Etudes Sci.},
  volume  = {49},
  year    = {1979},
  pages   = {5--233}
}

@article{Mo1,
  author  = {Moser, J\"urgen},
  title   = {On invariant curves of area-preserving mappings of an annulus},
  journal = {Nachr. Akad. Wiss. G\"ottingen Math.-Phys. Kl. II},
  year    = {1962},
  pages   = {1--20}
}

@article{R1,
  author  = {R\"ussmann, Helmut},
  title   = {{Kleine Nenner I: \"Uber invariante Kurven differenzierbarer Abbildungen eines Kreisringes}},
  journal = {Nachr. Akad. Wiss. G\"ottingen Math.-Phys. Kl. II},
  year    = {1970},
  pages   = {67--105}
}

@article{Y2,
  author  = {Yoccoz, Jean-Christophe},
  title   = {Conjugaison diff\'erentiable des diff\'eomorphismes du cercle dont le nombre de rotation v\'erifie une condition diophantienne},
  journal = {Ann. Sci. \'Ecole Norm. Sup. (4)},
  volume  = {17},
  year    = {1984},
  pages   = {333--359}
}

@incollection{yoc92,
  author    = {Yoccoz, Jean-Christophe},
  title     = {Travaux de {Herman} sur les tores invariants},
  booktitle = {S\'eminaire Bourbaki, Vol. 1991/92, Expos\'es 745--759},
  series     = {Ast\'erisque},
  number     = {206},
  year       = {1992},
  pages      = {311--344},
  publisher  = {Soci\'et\'e math\'ematique de France},
  note       = {Expos\'e no.~754, 34 pp.}
}

@article{Gotay,
  author={Gotay, Mark J.},
  title={{On coisotropic imbeddings of presymplectic manifolds}},
  journal={Proceedings of the American Mathematical Society},
  volume={84},
  number={1},
  pages={111--114},
  year={1982},
  doi={10.1090/S0002-9939-1982-0633290-X}
}

@article{CTdL,
  title={Contact type solutions and non-mixing of the 3D Euler equations},
  author={Cardona, Robert and de Lizaur, Francisco Torres},
  journal={arXiv preprint arXiv:2312.03514},
  year={2023}
}

@article{Sch,
  title={Asymptotic cycles},
  author={Schwartzman, Sol},
  journal={Annals of Mathematics},
  volume={66},
  number={2},
  pages={270--284},
  year={1957},
  publisher={JSTOR}
}

@article{Z,
    title={Generic density of periodic orbits of area-preserving maps on punctured surfaces},
    author={Zhou, Shaoyang},
    journal={arXiv preprint arXiv:2411.15429},
    year={2024}
    }

\Addresses

\end{document}